\documentclass[12]{amsart}

\usepackage{amsmath}
\usepackage{amsfonts}
\usepackage{rotating}
\usepackage{bbm}
\usepackage[all,cmtip]{xy}

\usepackage{amssymb}
\usepackage[latin1]{inputenc}
\usepackage{graphicx}
\usepackage{dsfont}
\usepackage{color}
\usepackage[hyperref]{hyperref}

\renewcommand{\phi}{\varphi}

\newcommand{\R}{{\mathbb{R}}}

\renewcommand{\P}{{\mathbb{P}}}

\advance\hoffset by -0.5 truecm
\usepackage{amssymb}
\newtheorem{Theorem}{Theorem}[section]
\newtheorem{Lemma}[Theorem]{Lemma}
\newtheorem{Corollary}[Theorem]{Corollary}

\newtheorem{Proposition}[Theorem]{Proposition}

\newtheorem{Definition}[Theorem]{Definition}
\newtheorem{Remark}[Theorem]{Remark}

\def\R{{\mathbb R}}

\def \0{\lambda_{0}}

\newcommand{\w}{\wedge}

\begin{document}
\title[Contacting the Moon]{The contact geometry of the restricted 3-body problem}


\author[P. Albers]{Peter Albers}
 \address{Department of Mathematics, Purdue University}
 \email {palbers@math.purdue.edu}

\author[U. Frauenfelder]{Urs Frauenfelder}
 \address{Department of Mathematics and Research Institute of Mathematics,
  Seoul National University}
 \email {frauenf@snu.ac.kr}

\author[O. van Koert]{Otto van Koert}
\address{Department of Mathematics and Research Institute of Mathematics,
  Seoul National University}
 \email {okoert@snu.ac.kr}

\author[G.P. Paternain]{Gabriel P. Paternain}
 \address{ Department of Pure Mathematics and Mathematical Statistics,
University of Cambridge,
Cambridge CB3 0WB, UK}
 \email {g.p.paternain@dpmms.cam.ac.uk}



\begin{abstract}
We show that the planar circular restricted three body problem
is of restricted contact type for all energies below the first critical value (action of the
 first Lagrange point) and for energies slightly above it. This opens up the
possibility of using the technology of Contact Topology to understand this
particular dynamical system.
\end{abstract}

\maketitle

\section{Introduction}

Starting with the fundamental contributions of Gromov and Floer,
holomorphic curve techniques had a great impact on our understanding
of the dynamics of Hamiltonian systems. The three body problem is a
 Hamiltonian system with a particular intriguing dynamics. During the
visit of the first two authors to the IAS in Princeton, Helmut Hofer
raised the question of whether holomorphic curve techniques can be applied
to the three body problem.

The power of holomorphic curve techniques relies on a particular
compactness result, called Gromov compactness. In order to prove
Gromov compactness the Hamiltonian system has to meet an additional
convexity assumption. Namely a Liouville vector field transverse to
energy hypersurfaces has to exist. We prove in this paper that this
assumption holds true for the planar circular restricted three body
problem for energy values below and slightly above the first
critical value.

The restricted three body problem assumes that two massive
primaries, which we refer to as the earth and the moon, rotate around
each other according to Kepler's law on ellipses. We are interested
in the dynamics of a third massless object called the satellite.
Since the satellite is massless it does not influence the
propagation of the primaries, however the primaries attract the
satellite according to Newton's law of gravitation. If the ellipses
are circles then the restricted three body problem in rotating
coordinates admits an integral, which we abbreviate by $H$. This
special case of the restricted three body problem is called the
circular restricted three body problem. If in addition the satellite
stays on the plane spanned by the earth and the moon the problem is
referred to as the planar circular restricted three body problem.

Energy hypersurfaces for the planar circular restricted three body
problems are three dimensional. However, they are not compact. This
is due to collisions of the satellite with one of the primaries.
However, it is well known in celestial mechanics that two body
collisions can be regularized. It was observed by Moser that the
regularized Kepler problem coincides with the geodesic flow of the
two sphere. In particular, energy hypersurfaces are diffeomorphic to
$\mathbb{R}P^3$. In Moser's regularization the position and momenta
variables are interchanged and the point at infinity on
the sphere corresponds to collision orbits whose velocity explodes.

For low energy levels $c \in \mathbb{R}$ the energy hypersurface
$\Sigma_c=H^{-1}(c)$ for the planar circular restricted three body
problem has three connected components. Indeed, the satellite either
has to stay close to the earth, close to the moon, or be far away.
We denote by $\Sigma_c^E$ and $\Sigma^M_c$ the connected component
close to the earth respectively close to the moon. Via Moser's
regularization these components can be compactified again to compact
manifolds $\overline{\Sigma}_c^E$ and $\overline{\Sigma}_c^M$ which
are both diffeomorphic to $\mathbb{R}P^3$.

On the axis between earth and moon there is a critical point of the
energy called the first Lagrange point, which we denote by $L_1$. If
the energy level $c$ becomes higher than the action of the first
Lagrange point $H(L_1)$, then the satellite is in principle able to
cross from the region around the earth to the region around the
moon. The energy hypersurface now consists of only two connected
components, one bounded which contains the regions around the earth
and the moon and an unbounded one. We abbreviate by $\Sigma_c^{E,M}$
the bounded component. Topologically this bounded component is just
the connected sum of the two bounded components we had before. In
particular, its regularization $\overline{\Sigma}_c^{E,M}$ is
diffeomorphic to $\mathbb{R}P^3 \# \mathbb{R}P^3$.

To state our main theorem we need the following convention.

\begin{Definition} {\rm Let $\Sigma$ be a $2n+1$-dimensional manifold. A Hamiltonian
structure on $\Sigma$ is a closed two-form whose kernel defines a
one-dimensional distribution. A \emph{contact form} on $\Sigma$ is a
1-form $\lambda$ such that $\lambda \wedge (d \lambda)^n$ is a
volume form. If $\omega$ is a Hamiltonian structure on $\Sigma$,
then a contact form $\lambda$ is called \emph{compatible} with
$\omega$ if $d\lambda=\omega$. }
\end{Definition}
\begin{Remark}{\rm 
Abraham and Marsden use in \cite{abraham-marsden} different
conventions. A Hamiltonian structure is in their language a contact
form, where a contact form in our sense corresponds to an exact
contact form.}
\end{Remark}
Moser's regularization is actually symplectic and therefore the
regularized hypersurfaces are naturally endowed with a Hamiltonian
structure. The main result of this paper is:
\\ \\
\textbf{Theorem A.} \emph{For $c<H(L_1)$ both connected components
$\overline{\Sigma}_c^E$ and $\overline{\Sigma}_c^M$ admit a
compatible contact form $\lambda$. Moreover, there exists
$\epsilon>0$ such that if $c \in (H(L_1),H(L_1)+\epsilon)$ the same
assertion holds true for $\overline{\Sigma}_c^{E,M}$.}
\begin{Remark}
{\rm The drawback of Theorem~A is that we only can prove the contact
condition until slightly above the first critical value. There is no
particular reason that the contact condition should fail for higher
energy values and we actually expect that for every energy level the
hypersurfaces are contact. However, this is an issue for further
research.}
\end{Remark}
\begin{Remark}
{\rm Moser regularization actually endows the compact components of the
energy hypersurfaces with a strong symplectic filling. In
particular, the contact structures obtained in Theorem~A are of restricted contact type and tight.}
\end{Remark}
Since there is only one tight contact structure on $\mathbb{R}P^3$
as well as on $\mathbb{R}P^3 \#\mathbb{R}P^3$ we obtain the
following corollary of Theorem~A.
\begin{Corollary}
For $c<H(L_1)$ the contact structures
$\big(\overline{\Sigma}_c^E,\ker \lambda\big)$ and
$\big(\overline{\Sigma}_c^M,\ker \lambda\big)$ coincide with the
tight $\mathbb{R}P^3$ and for $c \in (H(L_1),H(L_1)+\epsilon)$ the
contact structure $\big(\overline{\Sigma}_c^{E,M},\ker \lambda\big)$
coincides with the tight $\mathbb{R}P^3\#\mathbb{R}P^3$.
\label{cor:tight}
\end{Corollary}

We note that in general it is a very difficult problem to decide if a given
Hamiltonian structure admits a compatible contact form. A notorious concrete
case is that of Hamiltonian structures arising from Hamiltonians which contain
``magnetic terms'', i.e., terms which are linear in the momentum variables
(cf. \cite{CMP,CFP}). This is precisely the case of the planar circular restricted
three body problem which picks up a term of the form $p_1q_2-p_2q_1$ due to the
rotating frame.

Before we mention two applications of Theorem A we point out that the strong symplectic filling of the regularized energy hypersurface $\overline{\Sigma}_c^{E}$ is $T^*_{\leq1}S^2$, the unit co-disk bundle of $S^2$, and of $\overline{\Sigma}_c^{E,M}$ is the
boundary connected sum $T^*_{\leq 1}S^2 \natural T^*_{\leq 1}S^2$, see Remark \ref{rmk:bdry_connected_sum}.

The first  application concerns existence of periodic orbits.
It follows from the fact that the contact structures are tight, and Rabinowitz's result on existence of closed orbits on starshaped hypersurfaces \cite{rab}. The details of this derivation are explained by  Hofer in \cite{hofer} together with the fact that the lift of the tight contact structure on $\mathbb{R}P^3$ to $S^3$ is the unique tight contact structure on $S^3$.

\begin{Corollary} For any value $c<H(L_{1})$, the regularized planar circular restricted
three body problem has a closed orbit with energy $c$.
\end{Corollary}

\begin{Remark}\rm
The assertion of the previous corollary is trivially true for larger values of $c$ since the standard Weinstein model for the connected sum contains a periodic orbit.
\end{Remark}

The second application concerns leaf-wise intersection points. This notion is due to Moser \cite{Moser_A_fixed_point_theorem_in_symplectic_geometry}. We recall that if $\Sigma\subset(W,\omega)$ is a regular hypersurface then $\Sigma$ is foliated by the leaves of the characteristic line bundle $\ker\omega|_\Sigma$. We denote by $L_x$ the leaf through $x\in\Sigma$ and furthermore by $\mathrm{Ham}_c(W)$ the set of Hamiltonian diffeomorphisms generated by a Hamiltonian function with compact support. Then, by definition, $x\in\Sigma$ is a leaf-wise intersection point with respect to $\psi\in\mathrm{Ham}_c(W)$ if $\psi(x)\in L_x$.

\begin{Corollary} 
Let $\epsilon$ be the same as in Theorem A. Then for $c <H(L_1)+\epsilon, c\neq H(L_1)$ and any $\psi\in\mathrm{Ham}_c(T^*S^2)$ the regularized energy hypersurface $\overline{\Sigma}_c^{E,M}$ carries a leaf-wise intersection point with respect to $\psi$ .
\end{Corollary}

\begin{proof}
Since the contact structures are of restricted contact type we can apply the theorem by Cieliebak-Frauenfelder-Oancea \cite{Cieliebak_Frauenfelder_Oancea_Rabinowitz_Floer_homology_and_symplectic_homology} and obtain that the Rabinowitz Floer homology of the energy hypersurface as defined by Cieliebak-Frauenfelder \cite{Cieliebak_Frauenfelder_Restrictions_to_displaceable_exact_contact_embeddings} is isomorphic to the direct sum of symplectic homology and cohomology (except possibly in finitely many degrees.) By Cieliebak's \cite{Cieliebak_Handle_attaching_in_symplectic_homology_and_the_chord_conjecture} computation of the symplectic homology of a boundary-connected sum, see also McLean \cite[Theorem 2.20]{McLean_Lefschetz_fibrations_and_symplectic_homology} for an improvement concerning the ring structure, we know
$$
SH_*(T^*_{\leq 1}S^2 \natural T^*_{\leq 1}S^2)\cong SH_*(T^*_{\leq 1}S^2)\times SH_*(T^*_{\leq 1}S^2)\;.
$$
Combining this with the work of Viterbo \cite{Viterbo_Symplectic_topology_as_the_geometry_of_generating_functions}, Salamon-Weber \cite{Salamon_Weber_Floer_homology_and_heat_flow}, Abbondandolo-Schwarz \cite{Abbo_Schwarz_On_the_Floer_homology_of_cotangent_bundles}
$$
SH_*(T^*_{\leq 1}S^2)\cong H_*(\Lambda S^2)
$$
where $\Lambda S^2$ denotes the free loop space of $S^2$ we see that in both cases, $c<H(L_1)$ and $c <H(L_1)+\epsilon$, Rabinowitz Floer homology of the energy hypersurface is infinite dimensional, in particular, non-zero. Thus, work by Albers-Frauenfelder \cite[Theorem C]{Albers_Frauenfelder_Leafwise_intersections_and_RFH} implies the Corollary.
\end{proof}

\begin{Remark}\rm
We remark that Theorem A and all its corollaries are valid for any mass ratio between the two primaries.
\end{Remark}

\section{Elements of $3$-dimensional contact topology}
In the proof of Theorem~A it is convenient, although not necessary, to use some contact topology.
In doing so, we can regard the dynamical problem from another perspective.
Let us begin by defining some key notions.

\begin{Definition}
A contact $3$-manifold $(M,\xi)$ is called {\bf overtwisted} if $(M,\xi)$ admits an overtwisted disk, i.e.,
an embedded disk $\Delta$ such that $\xi_p=T_p\Delta$ for all $p\in \partial \Delta$.
Contact $3$-manifolds that are not overtwisted are called {\bf tight}.
\end{Definition}
The contact topological behavior of tight and overtwisted contact manifolds differs very much.
Roughly speaking, one could say that tight $3$-manifolds are much more rigid than overtwisted ones.
Finding tight contact structures is a hard problem in general, but there is a simple criterion due to Eliashberg and Gromov, see \cite{eliashberg} and \cite{gromov}.
\begin{Theorem}[Eliashberg, Gromov]
\label{thm:fillable_tight}
Suppose $(M,\xi)$ is a symplectically fillable $3$-manifold, then $(M,\xi)$ is tight.
\end{Theorem}
By symplectic fillability we mean the following.
\begin{Definition}
{\rm A symplectic manifold $(W,\omega)$ is called a {\bf strong symplectic filling} for $(M,\xi)$ if $\partial W=M$ and there is a Liouville vector field $X$ pointing outward along $M$ such that $\xi=\ker (i_X \omega)|_M$.}
\end{Definition}

Some uniqueness results are known about tight $3$-manifolds.
Note that the general situation can be much more complicated, but the following suffices for our needs.
\begin{Theorem}[Eliashberg]
\label{thm:unique_tight_structure_RP3}
The closed $3$-manifold $\R P^3$ admits a unique tight contact structure up to isotopy.
\end{Theorem}
This unique tight contact structure is, in fact, symplectically fillable; its filling is $T^*S^2$.
Alternatively, we can see $(\R P^3,\xi_0)$ as the contact manifold obtained from $(S^3,\xi_0)$ by dividing out by the antipodal map.

Finally, let us point out that tight manifolds admit a unique prime decomposition, see \cite{Ding_Geiges}.
\begin{Theorem}[Ding-Geiges]
\label{thm:unique_prime_decomposition}
Every non-trivial tight contact $3$-manifold $(M,\xi)$ is contactomorphic to a connected sum
$$
(M_1,\xi_1)\# \cdots \# (M_k,\xi_k)
$$
of finitely many prime tight contact $3$-manifolds.
The summands are unique up to order and contactomorphism.
\end{Theorem}

By prime we mean that a manifold is not decomposable in a non-trivial connected sum; non-triviality is needed to exclude connected sums with $(S^3,\xi_0)$.

\subsection{Weinstein model for connected sums and symplectic handles}
Let us briefly recall the Weinstein model for connected sums.
Consider $(\R^{2n+2},\omega_0=d \vec x\w d \vec y +dz\w dw)$, where $\vec x,\vec y \in \R^{n}$ and $z,w\in \R$.
Observe that the Liouville vector field
$$
X=\frac{1}{2}( \vec x \partial_{\vec x}+\vec y \partial_{\vec y})+2 z \partial_{z} -w \partial_{w}
$$
is transverse to level sets of $f=\vec x^2+\vec y^2+z^2-w^2$ provided $f\neq 0$.
This observation is the starting point to show that connected sums of contact manifolds carry a contact structure.

Let $(M_1,\xi_1)$ and $(M_2,\xi_2)$ be contact manifolds, and choose Darboux balls $D_i\subset M_i$.
To perform the contact connected sum, observe that we can also embed $D_1 \cup D_2$ in the two connected components of $f^{-1}(-1)$.
In fact, one could regard $f^{-1}(-1)$ as two disjoint Darboux balls embedded as contact hypersurfaces in $\R^{2n+2}$.
Now connect these two Darboux balls by ``interpolating" $f^{-1}(-1)$ to $f^{-1}(1)$.
This is the neck of the connected sum.
Note that this interpolation process is possible, since $f^{-1}(-1)$ and $f^{-1}(1)$ are close to each other for large $w$.
The Liouville vector field $X$ is transverse to this neck, and hence the connected sum is contact.

\begin{Remark}
{\rm Alternatively, one could use the Liouville flow to ``map" $f^{-1}(-1)$ to $f^{-1}(1)$.
In this process, some isotopy is still necessary, though.}
\end{Remark}

\subsection{Symplectic handles}
\label{sec:symplectic_handles}
Suppose we want to perform symplectic handle attachment.
To do so, consider the symplectic $1$-handle,
$$
SH:=
\{ -2 \leq f \leq 2 \}
,
$$
and let $W$ be a (not necessarily connected) symplectic cobordism $W$ with boundary $\partial W=M$.
Take $p_1$ and $p_2$ points in $M$ along which we shall attach the handle.

By Darboux's theorem we find neighborhoods $D_i$ of $p_i$ that are (strictly) contactomorphic to $(\R^{2n+1},\alpha_0)$.
That means that we can also find a strict contactomorphism $\phi_1 \cup \phi_2$ from $D_1\cup D_2$ to a compact subset of the two components of the level set $f^{-1}(-1)$.

Furthermore, $W$ is a symplectic cobordism, so a collar neighborhood of the boundary is symplectomorphic to $([-1,0]\times M,d (e^t\alpha) )$.
We can construct a symplectomorphism from a collar neighborhood of $D_i$ to a compact subset of $SH$ by gluing flow lines of the Liouville vector field on $W$ ($\partial_t$ in a collar neighborhood of the boundary) to flow lines of the Liouville vector field on $SH$.

In other words, the symplectomorphism is of the form
\begin{eqnarray*}
\psi_i: [-1,0]\times D_i \subset [-1,0]\times M & \longrightarrow & SH\\
(t,x) & \longmapsto & Fl^X_{t}(\phi_i(x))
.
\end{eqnarray*}

\begin{Remark}{\rm 
The above formula is not quite accurate, since the amount a flow lines changes the value of $f$ in $SH$ depends on the the precise starting point.
In other words, $\psi_i$ maps the collar neighborhood $[-1,0]\times D_i$ into a compact subset of $SH$, see Figure~\ref{fig:weinstein_surgery} for a sketch.}
\end{Remark}

Define $SH_c$ to be the compact subset of $SH$ consisting of point $x$ that lie on a flow line of $X$ through $\psi(0,D_i)$.
Define the symplectic manifold obtained by handle attachment by 
$$
\tilde W:=W \bigcup SH_c / \sim
,
$$
where we identify $(t,x)\in [-1,0] \times D_i \subset W$ with $p$ in $SH_c$ if and only if $p=\psi(t,x)$.
See Figure~\ref{fig:weinstein_surgery} for an illustration of the symplectic handle attachment.

\begin{figure}[htp]
  \centering
  \includegraphics[width=0.6\textwidth,clip]{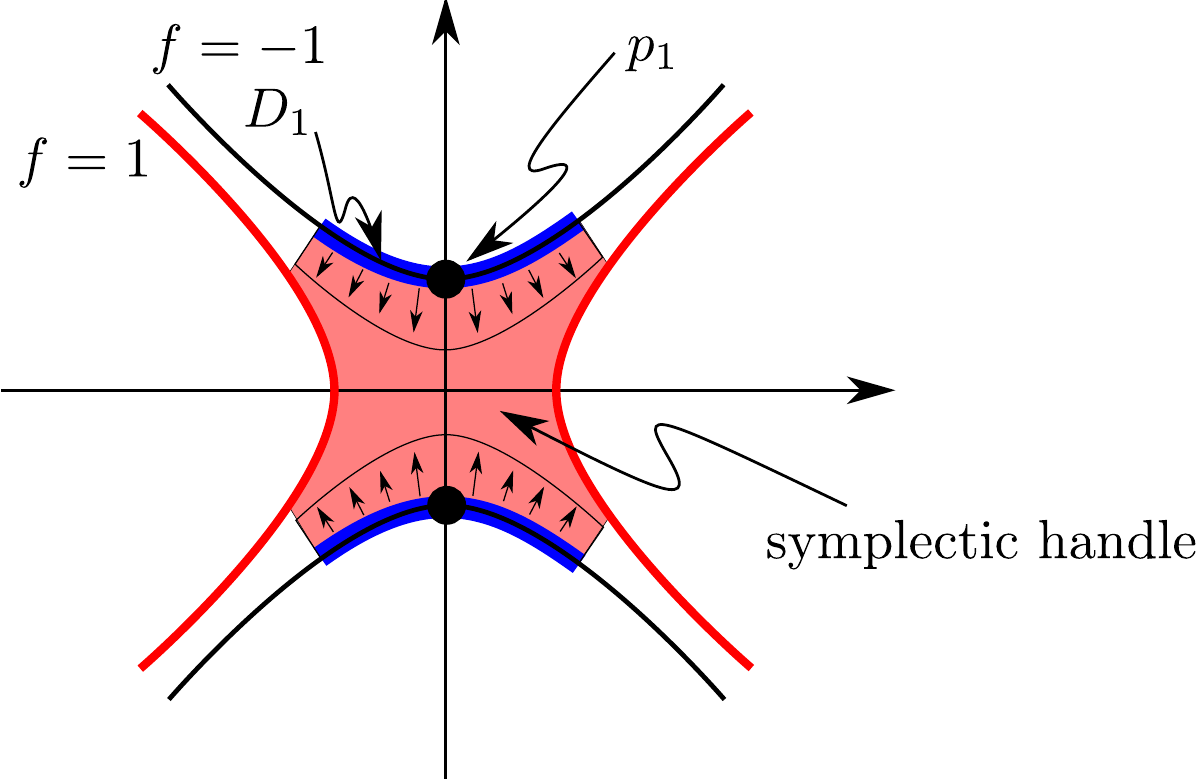}
  \caption{A symplectic handle}
  \label{fig:weinstein_surgery}
\end{figure}

\section{Moser regularization}
\label{section:reg}
In this section we recall the regularization introduced by J. Moser in \cite{moser}. This regularization will be used in subsequent sections, but here as a warm-up problem we apply it to the following integrable Hamiltonian on
$T^*(\R^2\setminus\{(0,0)\})$:
\begin{equation}
H(p,q)=\frac{1}{2}|p|^2-\frac{1}{|q|}+p_{1}q_{2}-p_{2}q_{1}.
\label{eq:hkr}
\end{equation}
This Hamiltonian corresponds precisely to the case when we set $\mu=0$
in (\ref{mh}) and it describes the Kepler problem in a rotating frame.

The first step is to change the independent variable $t$ to $s$ by setting
\[s=\int\frac{dt}{|q|}\]
and consider a new Hamiltonian $K$ defined by
\[H=\frac{K}{|q|}+k.\]
Here $k$ is the energy level of $H$ that we are interested in regularizing. 
The observation now is that for points at $K=0$ (which corresponds to $H=k$) one has
\[H_{p}=\frac{K_{p}}{|q|},\;\;\;\;\;\;H_{q}=\frac{K_{q}}{|q|}.\]
Then the orbits of $H$ with energy $k$ and time parameter $t$ correspond
to orbits in $K=0$ with time parameter $s$.
We now compute $K$ and we obtain:
\[K=\frac{1}{2}(|p|^2+1)|q|+|q|(p_{1}q_{2}-p_{2}q_{1})+(-k-1/2)|q|-1.\]
The second step consists in performing the canonical transformation $p=-x$ and $q=y$,
so that now we would like to understand the zero energy level of:
\begin{equation}
K=\frac{1}{2}(|x|^2+1)|y|+|y|(-x_{1}y_{2}+x_{2}y_{1})+(-k-1/2)|y|-1.
\label{eq:K}
\end{equation}

The last step is to introduce stereographic projection. Using the same notation
as in \cite{moser}, we let $\xi=(\xi_{0},\xi_{1},\xi_{2})$ be a point in $\mathbb{R}^3$
with norm one, so it represents a point in $S^2$. A tangent vector $\eta\in T_{\xi}S^2$
is written as $\eta=(\eta_0,\eta_1,\eta_2)$, with inner product $(\xi,\eta)=0$. 
We shall identify $TS^2$ with $T^*S^2$ using the standard metric on $S^2$.
The transformation induced by stereographic projection is described by the equations:
\begin{align}
x_{k}&=\frac{\xi_{k}}{1-\xi_{0}},\;\;\;\;\;k=1,2,\label{str1}\\
y_{k}&=\eta_{k}(1-\xi_0)+\xi_{k}\eta_0,\;\;\;\;\;k=1,2.\label{str2}
\end{align}
This is a symplectic transformation between $T^*\R^2$ and $T^*S^2$ and we also have the
relations:
\begin{equation}
|\eta|=\frac{(|x|^2+1)|y|}{2}=\frac{|y|}{1-\xi_0}.
\label{str3}
\end{equation}
We now apply these formulas to the Hamiltonian $K$ in (\ref{eq:K}) and we find after a calculation that our system goes over the hypersurface in $T^*S^2$ given by:
\[|\eta|+(1-\xi_{0})|\eta|(\eta_{1}\xi_{2}-\eta_{2}\xi_{1})+(1-\xi_{0})|\eta|(-k-1/2)=1.\]
To make this smooth at the zero section we consider the Hamiltonian
on $T^*S^2$
\begin{equation}
Q(\xi,\eta)=\frac{1}{2}|\eta|^2[1+(1-\xi_{0})(-k-1/2+\eta_{1}\xi_{2}-\eta_{2}\xi_{1})]^2
\label{eq:Q}
\end{equation}
and we are interested in the hypersurface $\Sigma_k:=Q^{-1}(1/2)$.
In a moment we will show that the hypersurfaces $\Sigma_k$ are all of contact type
for the relevant range of energies ($k<-3/2$), but first a remark.

\begin{Remark}{\rm If we had started without the term $p_{1}q_{2}-p_{2}q_{1}$, which
corresponds to the rotating frame, 
the Hamiltonian $Q$ would be
\[Q(\xi,\eta)=\frac{1}{2}|\eta|^2[1+(1-\xi_{0})(-k-1/2)]^2.\]
For $k=-1/2$, which is the energy considered by Moser in \cite{moser}, the system
is precisely the geodesic flow of the standard metric on $S^2$. 
For $k<0$ (the relevant range of energies for the Kepler problem) we also obtain a Riemannian metric on $S^2$ given by
\[g_{k}=\frac{1}{[1+(1-\xi_{0})(-k-1/2)]^2}\;g_{st}\]
where $g_{st}$ is the standard metric on $S^2$. An elementary calculation (that we omit)
shows that the curvature of $g_k$ is constant an equal to $-2k$, so up to scaling,
$g_k$ is isometric to $g_{st}$.

}
\end{Remark}

Returning to the Hamiltonian (\ref{eq:hkr}), we rewrite it as:
\[H(q,p)=\frac{1}{2}\left((p_1+q_2)^2+(p_2-q_1)^2\right)+U(q),\]
where $U$ is the effective potential. If we let $r=|q|$ then
\[U(r)=-1/r-r^2/2.\]
The boundary of the Hill's region (see \eqref{def:Hills_region} for the definition) with energy $k$ is given by
\[k+1/r+r^2/2=0\]
so we look at the cubic
\[r^3+2kr+2=0.\]
It follows that there is a compact Hill's region iff $k\leq -3/2$.
So for $k<-3/2$ the compact Hill's region is given by
$r\leq r_0$ where $r_0(k)$ is the smallest positive root of the cubic
($r_0(-3/2)=1$). Thus $k<-3/2$ is the range of energies for which $\Sigma_k$
has a compact connected component containing the zero section. We now show:

\begin{Proposition} $\Sigma_k$ is starshaped for $k<-3/2$.
\end{Proposition}

\begin{proof} We shall check that the Liouville vector field $X=\eta\partial_{\eta}$
is transverse to $\Sigma_k$. Since everything is rotationally symmetric about the
 $\xi_0$-axis, it is enough to consider the case of $\xi_1=0$ and $\xi_0^2+\xi_2^2=1$.
So we look at $\Sigma_k$ over these points only:
\begin{equation}
|\eta|[1+(1-\xi_{0})(a+\eta_{1}\xi_{2})]=1
\label{eq:sigma}
\end{equation}
where $a:=-1/2-k$. Now for $\lambda$ near $1$ we let
\[f(\lambda):=\lambda|\eta|[1+(1-\xi_{0})(a+\lambda\eta_{1}\xi_{2}).\]
We wish to check that $f'(1)\neq 0$ for points satisfying (\ref{eq:sigma}).
After a calculation we find
\[f'(1)=\frac{1+(1-\xi_0)(a+2\eta_{1}\xi_2)}{1+(1-\xi_0)(a+\eta_{1}\xi_2)}.\]
Thus $f'(1)=0$ iff
\[\eta_{1}\xi_2=-\frac{a}{2}-\frac{1}{2(1-\xi_0)}=-\frac{1+(1-\xi_0)a}{2(1-\xi_0)}.\]
Note that at the point with $\xi_0=1$ there is nothing to check and without
loss of generality we may assume $\xi_2>0$.
Using (\ref{eq:sigma}) we see that still assuming $f'(1)=0$
\[|\eta|\left(\frac{1+(1-\xi_0)a}{2}\right)=1.\]
Clearly $|\eta|\geq |\eta_{1}|$, thus
\[\frac{(1+(1-\xi_0)a)^2}{4\xi_{2}(1-\xi_0)}\leq 1.\]
Now, if we let $\xi_{0}=\cos\theta$ and $\xi_{2}=\sin\theta$ we can rewrite
the above as
\[g_a(\theta):=\frac{(1+(1-\cos\theta)a)^2}{4\sin\theta(1-\cos\theta)}\leq 1.\]

\noindent{\bf Claim.} If $a>1$, then $g_a >1$.

\medskip

We introduce the function
$$h \colon (-1,1) \to \mathbb{R}, \quad t \mapsto
\frac{(2-t)^4}{16(1-t^2)(1-t)^2}=-\frac{(t-2)^4}{16(t+1)(t-1)^3}.$$
Note that for $\theta \in (0,\pi)$ we have
\begin{equation}\label{hg}
h(\cos\theta)=g_1(\theta)^2.
\end{equation}
The derivative of $h$ is given at $t \in (-1,1)$ by the formula
$$h'(t)=-\frac{3t(t-2)^3}{8(t+1)^2(t-1)^4}.$$
In particular, the only zero of $h'$ in $(-1,1)$ is located at
$t=0$. Because $\lim_{t \to \pm 1} h(t)=\infty$ we conclude that $h$
attains its global minimum at $t=0$. Since $h(0)=1$ we deduce
\begin{equation}\label{hg2}
h \geq 1.
\end{equation}
Combining (\ref{hg}) and (\ref{hg2}) we conclude that
$$g_1 \geq 1.$$
Since for $a>1$ it holds that $g_a>g_1$ the assertion of the Claim
follows.

It follows that for $k<-3/2$, $f'(1)\neq 0$ and we have transversality
of the Liouville vector field with $\Sigma_k$.

\end{proof}

\section{The circular planar restricted 3-body problem}

Assume that we have two massive bodies called the primaries, which
we will denote as Earth and Moon, which propagate according to
Newton's law of gravitation on circles around their common center of
mass. If $m_E$ is the mass of the earth and $m_M$ is the mass of the
moon we denote by
$$\mu=\frac{m_M}{m_E+m_M} \in [0,1]$$
the proportion of the mass of the moon on the total mass. In the
following we normalize our problem such that the total mass
satisfies
$$m_E+m_M=1.$$
In an inertial plane spanned by the earth and the moon the position
of the earth becomes
$$E(t)=\big(\mu \cos t, \mu \sin t\big)$$
and the position of the moon
$$M(t)=\big(-(1-\mu)\cos t,-(1-\mu)\sin t\big).$$
Suppose further that $S$ is a third massless body moving in the
plane spanned by the earth and moon on which the gravitational force
of the two primaries act. The body $S$ is referred to as the
satellite. Since the satellite is massless it does not influence the
two primaries. For $t \in \mathbb{R}$ the time-dependent Hamiltonian
$$H^i_t \colon \big(\mathbb{R}^2
\setminus \{E(t),M(t)\}\big)\times \mathbb{R}^2 \to \mathbb{R}$$ for
the satellite is given at  $(q,p) \in \big(\mathbb{R}^2 \setminus
\{E(t),M(t)\}\big)\times \mathbb{R}^2$ by
$$H^i_t(q,p)=\frac{1}{2}|p|^2-\frac{1-\mu}{|q-E(t)|}-
\frac{\mu}{|q-M(t)|},$$ i.e.~the sum of the kinetic and potential
energy of the satellite. Here the superscript $i$ indicates that we
consider an inertial coordinate system. Instead of considering an
inertial coordinate system we can alternatively also study the
propagation of the satellite in moving or synodical coordinates
where the earth and moon are at rest. We choose the coordinates in
such way that the position of the earth and moon satisfy
$$E=\big(\mu,0\big), \quad M=\big(-(1-\mu),0\big).$$
The Hamiltonian in the synodical coordinate system for the satellite
is given by
$$H \colon \big(\mathbb{R}^2 \setminus \{E,M\}\big) \times
\mathbb{R}^2 \to \mathbb{R}$$ satisfying
\begin{equation}
H(q,p)=\frac{1}{2}|p|^2-\frac{1-\mu}{|q-E|}-
\frac{\mu}{|q-M|}+p_1q_2-p_2q_1.
\label{mh}
\end{equation}
The transition from the inertial
coordinate system to the moving coordinate system requires a
time-dependent transformation and we refer to \cite{abraham-marsden}
for a derivation of it. It is an amazing fact that in the rotating
coordinate system the Hamiltonian becomes autonomous,
i.e.~independent of time. In particular, the Hamiltonian $H$ is
preserved under the flow. This observation goes back to Jacobi and
the integral $-2H$ is traditionally called the Jacobi integral.

We further note that the Hamiltonian in the synodical coordinate system
is not anymore of the simple form kinetic plus potential energy but
contains the third term $p_1q_2-p_2q_1$ which is the moment map of
the $S^1$-action on the cotangent bundle of $\mathbb{R}^2$ induced
from rotations of the base. One might think of this term as
``rotational energy" governing the centrifugal and Coriolis forces of
the rotating system. In the case where the Hamiltonian is just the
sum of kinetic and potential energy it is well-known that energy
hypersurfaces are of contact type. However, because of the rotational term
the question of whether energy hypersufaces of the Hamiltonian $H$ are
contact is far from obvious. Indeed, formally the rotational forces
are nothing but the Lorentz force of a magnetic field (after switching to the 
effective potential, see below)
and we refer to \cite{CFP,CMP} for examples of energy hypersurfaces which are not
contact in the presence of a magnetic field.

If we complete the squares in the formula for the Hamiltonian $H$ we
end up with the expression
$$H(q,p)=\frac{1}{2}\Big((p_1+q_2)^2+(p_2-q_1)^2\Big)
-\frac{1-\mu}{|q-E|}-\frac{\mu}{|q-M|}-\frac{1}{2}|q|^2.$$ The
function
$$U \colon \mathbb{R}^2 \setminus \{E,M\} \to \mathbb{R}$$
given by
$$U(q)=-\frac{1-\mu}{|q-E|}-\frac{\mu}{|q-M|}-\frac{1}{2}|q|^2$$
is called the \emph{effective potential}. So we obtain the following
condensed formula for the Hamiltonian
\begin{equation}\label{hami}
H(q,p)=\frac{1}{2}\Big((p_1+q_2)^2+(p_2-q_1)^2\Big)+U(q).
\end{equation}
This expression of the Hamiltonian shows clearly the presence of 
an exact magnetic field with primitive given by the 1-form $q_{2}dq_1-q_{1}dq_2$.

The effective potential $U$ has five critical points
$(\ell^1,\ldots, \ell^5)$ called \emph{Lagrangian points}. The first
three critical points, which were already discovered by Euler, are
called collinear Lagrangian points since they lie on the axis
spanned by earth and moon. The first critical point $\ell^1$ lies in
between earth and moon, $\ell^2$ on the opposite side of the moon
and $\ell^3$ on the opposite side of the earth. The three collinear
Lagrange points are saddle points of the effective potential. The
two maxima of $U$ lie at $\ell^4$ and $\ell^5$ which are called the
triangular Lagrange points since both of them span together with the
two primaries an equilateral triangle. Note that if $R \colon
\mathbb{R}^2 \to \mathbb{R}^2$ is the reflection along the first
axis then $R(\ell^4)=\ell^5$ which is obvious from the fact that the
effective potential is invariant under $R$. We remark that the
triangular Lagrangian points are also sometimes referred to as
Trojan points since in the Sun-Jupiter system they correspond to the
locations of the Trojan asteroids.

There is a one-to-one correspondence between critical points of the
effective potential and critical points of the Hamiltonian. Namely
for $j \in \{1, \ldots, 5\}$ abbreviate
$$L_j=(\ell^j_1,\ell^j_2,-\ell^j_2,\ell^j_1) \in \mathbb{R}^4.$$
Then $L_1, \ldots, L_5$ are the five critical points of $H$. In case
where the mass of the moon is less than the mass of the earth,
i.e.~$\mu <\frac{1}{2}$ the critical points are ordered by action in
the following way
$$H(L_1)<H(L_2)<H(L_3)<H(L_4)=H(L_5).$$
If $c \in \mathbb{R}$ we abbreviate the energy hypersurface for the
Hamiltonian $H$ at energy level $c$ by
$$\Sigma_c=H^{-1}(c).$$
We further abbreviate by
$$\pi \colon \big(\mathbb{R}^2 \setminus \{E,M\}\big) \times
\mathbb{R}^2 \to \mathbb{R}^2 \setminus \{E,M\}$$ the projection
along $\mathbb{R}^2$. The \emph{Hill's region} is defined as
\begin{equation}\label{def:Hills_region}
\mathcal{K}_c=\pi(\Sigma_c) \subset \mathbb{R}^2 \setminus\{E,M\}.
\end{equation}
This is the region in position space where the satellite
of energy $c$ is allowed to stay. Note that since the first two
terms in (\ref{hami}) can attain in each fiber arbitrary nonnegative
values, we can alternatively define the Hill's region as
\begin{equation}\label{hire}
\mathcal{K}_c=\big\{q \in \mathbb{R}^2 \setminus \{E,M\}: U(q) \leq
c\big\}.
\end{equation}
 If $c<H(L_1)$, then $\mathcal{K}_c$ contains three
connected components, an unbounded one and two bounded ones in whose
closure lies the earth respectively the moon. We denote the bounded
components by $\mathcal{K}_c^E$ and $\mathcal{K}^M_c$ such that
$$E \in \mathrm{cl}(\mathcal{K}^E_c), \quad M \in
\mathrm{cl}(\mathcal{K}^M_c).$$ We further abbreviate by
$$\Sigma^E_c =\pi^{-1}\big(\mathcal{K}^E_c\big)\cap \Sigma_c, \quad
\Sigma^M_c=\pi^{-1}\big(\mathcal{K}^M_c\big)\cap\Sigma_c$$ the corresponding
connected components of the energy hypersurface $\Sigma_c$.

\section{Below the first critical level}
\label{section:below}
Assume that $c<H(L_1)$ and recall that $\Sigma^M_c \subset
H^{-1}(c)$ is the connected component of the energy hypersurface
around the moon. We introduce the Liouville vector field
\begin{equation}
\label{eq:Liouville_field_unregularized}
X=(q-M)\frac{\partial}{\partial q}.
\end{equation}
The main result of this section is
\begin{Proposition}\label{transverse}
$X$ intersects $\Sigma^M_c$ transversally.
\end{Proposition}
To prove the proposition we need three Lemmas. The Lemmas need several calculations 
the details of which are relegated to the Appendix. To formulate the
first Lemma we make the convention that $(\rho,\theta)$ are lunar
polar coordinates, i.e.~polar coordinates centered at the moon, such
that the effective potential becomes
$$U(\rho,\theta)=-\frac{\mu}{\rho}-\frac{1-\mu}{\sqrt{\rho^2-2\rho \cos
\theta+1}}-\frac{1}{2}\rho^2+\rho\cos\theta(1-\mu)-\frac{1}{2}(1-\mu)^2.$$
For $\rho \in (0,1)$ we define
$$U_\rho=U(\rho,\cdot) \in C^\infty(S^1,\mathbb{R}).$$
The first Lemma we need is the following assertion.
\begin{Lemma}\label{tra1}
For $\rho \in (0,1)$ the function $U_\rho$ attains its minimum at
$\theta=0$.
\end{Lemma}
\begin{proof} The derivative of $U_\rho$ is given by
$$U'_\rho(\theta)=(1-\mu)\rho\sin\theta
\bigg(\frac{1}{\big(\rho^2-2\rho\cos\theta+1\big)^{\frac{3}{2}}}-1\bigg).$$
We see that $\theta=0$ and $\theta=\pi$ are critical points of
$U_\rho$. The second derivative at these points is given by
$$U''_\rho(0)=(1-\mu)\bigg(\frac{1}{(1-\rho)^3}-1\bigg)>0$$
respectively
$$U''_\rho(\pi)=-(1-\mu)\bigg(\frac{1}{(1+\rho)^3}-1\bigg)>0.$$
We conclude that $U_\rho$ attains a local minimum at $0$ and $\pi$.
The remaining critical points of $U_\rho$ necessarily satisfy the
condition $\cos\theta=\frac{\rho}{2}$. In the two arcs bounded by
the minima there is precisely one such critical point, hence it has
to be a maximum. Therefore the only local minima lie at $\theta=0$
and $\theta=\pi$. We compare the values of $U_\rho$ at them and
obtain
$$
U_\rho(0)-U_\rho(\pi)=-\frac{2(1-\mu)\rho^3}{(1-\rho)(1+\rho)}<0.
$$
We conclude that $U_\rho$ attains its global minimum at $\theta=0$.
\end{proof}
Abbreviate
$$d:=|\ell^1-M|$$
the distance from the moon to the first Lagrangian point, i.e. the
critical point of the effective potential which lies on the
earth-moon axis between the two primaries. We further introduce
$$B=\{q \in \mathbb{R}^2: |q-M| \leq d\}$$
the ball centered at the moon of radius $d$. In particular, $\ell^1$
lies on the boundary of $B$. As a corollary of the previous Lemma we
obtain the following bound on the component of the Hill's region
around the moon. We denote by $\mathrm{int}$ the interior of a set.

\begin{Corollary}\label{cortra}
$\mathcal{K}^M_c \subset \mathrm{int}(B)$.
\end{Corollary}
\begin{proof}
We first show that
\begin{equation}\label{hill}
\mathcal{K}_c \cap \partial B=\emptyset.
\end{equation}
To see that pick $(d,\theta) \in \partial B$. We estimate using
Lemma~\ref{tra1}
\begin{equation}\label{boundary}
U(d,\theta)=U_d(\theta) \geq U_d(0)=U(\ell^1)=H(L^1)>c.
\end{equation}
By the description of the Hill's region given in (\ref{hire}) we see
that $(\rho,\theta)$ does not lie in $\mathcal{K}_c$ and therefore
(\ref{hill}) follows. We deduce that both sets $\mathcal{K}^M_c \cap
B$ and $\mathcal{K}^M_c \cap B^c$ are open and closed. Since $M$
lies in the closure of $\mathcal{K}^M_c$, the set $\mathcal{K}^M_c
\cap B$ is nonempty. Hence, since $\mathcal{K}^M_c$ is connected, we
conclude that $\mathcal{K}^M_c \cap B^c=\emptyset$. This finishes
the proof of the Corollary.
\end{proof}
The next Lemma we need in order to prove
Proposition~\ref{transverse} is the following assertion.
\begin{Lemma}\label{tra2}
For every $q \in B-\{M,\ell^1\}$ it holds that $\frac{\partial
U}{\partial \rho}(q) > 0$.
\end{Lemma}
\begin{proof}
We prove the Lemma in three steps.
\\ \\
\textbf{Step~1: }\emph{If $\rho \in (0,d)$, then $\frac{\partial
U}{\partial \rho}(\rho,0)>0$, i.e.~the assertion of the Lemma holds
on the earth-moon axis between the moon and the first Lagrange
point.}
\\ \\
We consider the function $u \in C^\infty\big((0,1),\mathbb{R}\big)$
given by
$$u(\rho)=U(\rho,0),\quad \rho \in [0,\infty).$$
Explicitly, the function $u$ is given by
\begin{equation}\label{fuu}
u(\rho)=-\frac{\mu}{\rho}-
\frac{1-\mu}{1-\rho}-\frac{1}{2}(\rho-1+\mu)^2.
\end{equation}
Its second derivative satisfies for $\rho \in (0,1)$
$$u''(\rho)=-\frac{2\mu}{\rho^3}-\frac{2(1-\mu)}{(1-\rho)^3}-1<0.$$
In particular, $u$ is strictly concave. Since the Lagrange point is
a critical point of $U$, we have $u'(d)=0$, from which the assertion
of Step~1 follows.
\\ \\
\textbf{Step~2: }\emph{For $\rho \in (0,1)$ let $V_\rho \in
C^\infty(S^1,\mathbb{R})$ be the function which is given for $\theta
\in S^1$ by  $V_\rho(\theta)=\frac{\partial U}{\partial
\rho}(\rho,\theta)$. Then $V_\rho$ attains its unique minimum at
$\theta=0$.}
\\ \\
We first compute the derivative of $V_\rho$ to be
$$V'_\rho(\theta)=
\frac{\partial^2 U}{\partial \rho \partial \theta}(\rho,\theta)
=(1-\mu)\sin\theta\Bigg(\frac{-2\rho^2+\rho\cos\theta+1}
{\big(\rho^2-2\rho\cos\theta+1\big)^{\frac{5}{2}}}-1\Bigg).$$ We
first examine the two critical points at $\theta=0$ and
$\theta=\pi$. The second derivatives at these points compute to be
$$V''_\rho(0)=(1-\mu)\Bigg(\frac{-2\rho^2+\rho+1}{(1-\rho)^5}-1\Bigg)
=(1-\mu)\Bigg(\frac{2\rho+1}{1-\rho)^4}-1\Bigg)>0$$ and
$$V''_\rho(\pi)=-(1-\mu)\Bigg(\frac{-2\rho^2-\rho+1}{(1+\rho)^5}-1\Bigg)>0$$
so that we conclude that $V_\rho$ has strict local minima at
$\theta=0$ and $\theta=\pi$.

We next show that there are no other local minima of $V_\rho$ than
$\theta=0$ and $\theta=\pi$. For this purpose we have a closer look
at the function $f \in C^\infty([-1,1],\mathbb{R})$ given by
$$f(\tau)=\frac{-2\rho^2+\rho\tau+1}
{\big(\rho^2-2\rho\tau+1\big)^{\frac{5}{2}}}-1,\quad \tau \in
[0,1].$$ As we just seen above we have
$$f(1)>0, \quad f(-1)<0.$$
Moreover, the function $f$ is strictly monotone. Hence there exists
a unique $\tau_0 \in (-1,1)$ such that $f(\tau_0)=0$. Critical
points of $V_\rho$ different from $\theta=0$ or $\theta=\pi$ are
precisely points $\theta \in S^1$ which meet the condition $\cos
\theta=\tau_0$. Hence there is precisely one critical point in the
interval $(0,\pi)$ and $(\pi, 2\pi)$. Since $V_\rho$ has local
minima at $\theta=0$ and $\theta=\pi$ we conclude that these other
critical points correspond to local maxima. Hence $\theta=0$ and
$\theta=\pi$ are the only local minima.

To determine the global minimum of $V_\rho$ it suffices now to
compare the values at $\theta=0$ and $\theta=\pi$. We obtain
$$V_\rho(0)-V_\rho(\pi)=
2(1-\mu)\frac{\rho^4-3\rho^2}{(1-\rho)^2(1+\rho)^2}<0.$$ Hence
$\theta=0$ is the unique global minimum for $V_\rho$.
\\ \\
\textbf{Step~3: }\emph{Proof of the Lemma.}
\\ \\
By Step~1 it suffices to show the Lemma for $\theta \neq 0$. Hence
using Step~1 again together with Step~2 we get for $\rho \in (0,d]$
and $\theta \in S^1$ the estimate
$$\frac{\partial U}{\partial \rho}(\rho,\theta)=V_\rho(\theta)
> V_\rho(0)=\frac{\partial U}{\partial \rho}(\rho,0) \geq 0.$$ This
finishes the proof of the Lemma.
\end{proof}
We continue to denote by $B$ the ball whose radius is the distance
from the moon to the first Lagrange point and which is centered at
the moon.
\begin{Lemma}\label{tra3}
For each $q \in B-\{M\}$, it holds that $\frac{\partial^2
U}{\partial \rho^2}(q) \leq -1$.
\end{Lemma}
\begin{proof}
We first compute
$$
\frac{\partial^2 U}{\partial \rho^2}(\rho,\theta)=
-\frac{2\mu}{\rho^3}-\frac{1-\mu}
{(\rho^2-2\rho\cos\theta+1)^{\frac{5}{2}}}\Big(2\rho^2-4\rho\cos\theta-1+3\cos^2\theta\Big)-1
$$
We introduce the function $W \in C^\infty(B-\{M\},\mathbb{R})$ which
is given in lunar polar coordinates by the formula
$$W(\rho,\theta)=
-\frac{2\mu}{\rho^3}-\frac{1-\mu}
{(\rho^2-2\rho\cos\theta+1)^{\frac{5}{2}}}\Big(2\rho^2-4\rho\cos\theta-1+3\cos^2\theta\Big).$$
Hence the assertion of the Lemma is equivalent to
$$W(q) \leq 0, \quad \forall\,\,q \in B-\{M\}.$$
We prove this assertion in four Steps.
\\ \\
\textbf{Step~1: } \emph{The distance $d$ from the moon to the first
Lagrange point is related to $\mu$ by the following formula
\begin{equation}
\label{eq:mu_Lagrange_point}
\mu=-\frac{d^5-3d^4+3d^3}{d^4-2d^3-d^2+2d-1}.
\end{equation}
}
The distance $d$ is given as the unique critical point in the
interval $[0,1]$ of the function $u$ introduced in (\ref{fuu}). The
derivative of $u$ at $\rho \in (0,1)$ is
$$u'(\rho)=\frac{\mu}{\rho^2}-\frac{1-\mu}{(1-\rho)^2}+1-\rho-\mu,$$
which gives the relation between $\mu$ and $d$ as claimed.
\\ \\
\textbf{Step~2: } \emph{For $\rho \in (0,1)$ let $W_\rho=W(\rho,
\cdot) \in C^\infty(S^1,\mathbb{R})$. Let $\vartheta \in S^1$ be a
point where $W_\rho$ attains its maximum. Then $\vartheta$ satisfies
$$
\cos\vartheta=\frac{\rho^2-1+\sqrt{-\rho^4+\rho^2+1}}{\rho}.
$$}
We first compute
$$W'_\rho(\theta)=\frac{3(1-\mu)\sin\theta}{(\rho^2-2\rho\cos\theta+1)^{\frac{7}{2}}}\bigg(\rho\cos^2\theta
+2(1-\rho^2)\cos\theta+\rho(2\rho^2-3)\bigg).$$ We see that the
function $W_\rho$ has critical points at $\theta=0$ and
$\theta=\pi$. We claim that both of these critical points are local
minima. We first check this for $\theta=0$. The second derivative of
$W_\rho$ at $\theta=0$ is given by
$$W''_\rho(0)=\frac{6(1-\mu)(\rho^3-\rho^2-\rho+1)}{(1-\rho)^7}
=\frac{6(1-\mu)(\rho+1)}{(1-\rho)^5}>0.
$$
This shows that $\theta=0$ is a local minimum for $W_\rho$.
Similarly, we get for $\theta=\pi$
$$W''_\rho(\pi)=-\frac{6(1-\mu)(\rho^3+\rho^2-\rho-1)}{(1-\rho)^7}
=\frac{6(1-\mu)(1+\rho)^2}{(1-\rho)^6}>0$$ implying that
$\theta=\pi$ is a local minimum as well. Therefore the local maximum
has to satisfy
$$\rho\cos^2\vartheta
+2(1-\rho^2)\cos\vartheta+\rho(2\rho^2-3)=0.$$ In particular,
$$
\cos\vartheta=\frac{\rho^2-1\pm\sqrt{-\rho^4+\rho^2+1}}{\rho}.
$$
Since
$$\frac{\rho^2-1-\sqrt{-\rho^4+\rho^2+1}}{\rho}<-1
$$
the assertion of Step~2 follows.
\\ \\
\textbf{Step~3: }\emph{The function $W$ attains its maximum at the
boundary of $B$.} \\ \\
We argue by contradiction and assume that $(\rho_0,\theta_0)$ is a
point in the interior of $B$ at which $W$ attains its maximum. We
first observe that
\begin{equation}\label{contr1}
0=\frac{\partial W}{\partial \rho}(\rho_0,\theta_0)=
\frac{6\mu}{\rho_0^4}+3(1-\mu)(\cos\theta_0-\rho_0)\frac{3+4\rho_0\cos\theta_0
-2\rho_0^2-5\cos^2\theta_0}{(\rho_0^2-2\rho_0\cos\theta_0+1)^{\frac{7}{2}}}.
\end{equation}
Moreover, at $\theta_0$ the function $W_{\rho_0}$ attains its
maximum, so that we obtain from Step~2 the equality
\begin{equation}\label{contr2}
\cos\theta_0=\frac{\rho_0^2-1+\sqrt{-\rho_0^4+\rho_0^2+1}}{\rho_0}.
\end{equation}
In particular, we have
\begin{equation}\label{contr3}
\cos\theta_0-\rho_0>0.
\end{equation}
Moreover, taking once more advantage of equality (\ref{contr2}), we
obtain
\begin{eqnarray*}
3+4\rho_0\cos\theta_0-2\rho_0^2-5\cos^2\theta_0 &=&
\frac{2\rho_0^4+4\rho_0^2-10 +(10-6\rho_0^2)
\sqrt{-\rho_0^4+\rho_0^2+1}}{\rho_0^2}\\
&=&\frac{(10-6\rho_0^2)^2(-\rho_0^4+\rho_0^2+1)-(2\rho_0^4+4\rho_0^2-10)^2}
{\rho_0^2\Big(10-2\rho_0^4-4\rho_0^2+(10-6\rho_0^2)
\sqrt{-\rho_0^4+\rho_0^2+1}\Big)}\\
&=&\frac{20\big(3-8\rho_0^2+7\rho_0^4-2\rho_0^6\big)}
{10-2\rho_0^4-4\rho_0^2+(10-6\rho_0^2)\sqrt{-\rho_0^4+\rho_0^2+1}}\\
&=&\frac{40(\rho_0-1)^2\big(\frac{3}{2}-\rho_0^2\big)}
{10-2\rho_0^4-4\rho_0^2+(10-6\rho_0^2)\sqrt{-\rho_0^4+\rho_0^2+1}}\\
&>&0.
\end{eqnarray*}
Together with (\ref{contr3}) this contradicts (\ref{contr1}). This
finishes the proof of Step~3.
\\ \\
\textbf{Step~4: }\emph{Proof that $W$ is nonpositive.}
\\ \\
Since $W(q)$ tends to $-\infty$ when $q$ tends to the moon, we
conclude that there exists a point $(\rho_0,\theta_0) \in B$ at
which $W$ attains its maximum. By Step~3 we get that
$$\rho_0=d.$$
By Step~2 we conclude that
$$\cos\theta_0=\frac{d^2-1+\sqrt{-d^4+d^2+1}}{d}.$$
Talking also advantage of Step~1 we obtain
\begin{eqnarray*}
W(\rho_0,\theta_0) &=&\frac{2d^2\big(d^2-3d+3\big)
\big(3-d^2-2\sqrt{-d^4+d^2+1}\big)^{\frac{5}{2}}}
{\big(d^4-2d^3-d^2+2d-1\big)
\big(3-d^2-2\sqrt{-d^4+d^2+1}\big)^{\frac{5}{2}} d^2}\\
& &-\frac{\big(d^5-2d^4+d^3-d^2+2d-1\big)
\big(6-2d^4-(6-2d^2)\sqrt{-d^4+d^2+1}\big)}
{\big(d^4-2d^3-d^2+2d-1\big)
\big(3-d^2-2\sqrt{-d^4+d^2+1}\big)^{\frac{5}{2}} d^2}.
\end{eqnarray*}
The righthand side is now a function which only depends on $d \in
[0,1]$. Plotting this function reveals that it is always negative.
An analytic argument is included in the appendix.

\end{proof}
Armed with the previous three Lemmas we are now in position to prove
the main result of this section.
\\ \\
\textit{Proof of Proposition~\ref{transverse}.}$\;$ In lunar polar
coordinates the Liouville vector field $X$ is given by
$$X=\rho \frac{\partial}{\partial \rho}.$$
By (\ref{hami}) the Hamiltonian in lunar polar coordinates is given
by
$$H(\rho,\theta)=\frac{1}{2}\Big((p_1+\rho \sin \theta)^2+
(p_2-\rho \cos \theta+1-\mu)^2\Big)+U(\rho,\theta).$$ To prove the
proposition we show that
$$X(H)|_{\Sigma^M_c}>0.$$
We first estimate using the Cauchy-Schwarz inequality
\begin{eqnarray*}
X(H)&=&\rho\sin \theta(p_1+\rho \sin \theta)+\rho \cos \theta
(p_2-\rho \cos \theta+1-\mu)+\rho \frac{\partial U}{\partial \rho}
\\
&\geq&\rho \frac{\partial U}{\partial \rho}-\sqrt{\rho^2\cos^2\theta
+\rho^2 \sin^2 \theta}\sqrt{(p_1+\rho \sin \theta)^2+ (p_2-\rho\cos
\theta+1-\mu)^2}\\
&=&\rho \frac{\partial U}{\partial \rho}-\rho\sqrt{2(H-U)}.
\end{eqnarray*}
If we restrict this to the energy hypersurface we obtain
$$X(H)|_{\Sigma^M_c} \geq \rho\Bigg(\frac{\partial{U}}{\partial \rho}
-\sqrt{2(c-U)}\Bigg)\Bigg|_{\Sigma^M_c}.$$ The righthand side is
independent of the momentum variables so that it is sufficient to
check its positivity on the Hill's region. Hence we are left with
showing
\begin{equation}\label{positiv}\Bigg(\frac{\partial{U}}{\partial \rho}
-\sqrt{2(c-U)}\Bigg)\Bigg|_{\mathcal{K}^M_c}>0.
\end{equation}
We pick a point $(\rho,\theta) \in \mathcal{K}^M_c$. By
Corollary~\ref{cortra} we have $\rho < d$. Moreover, by
(\ref{boundary}) we have that $U(d,\theta)>c$. Since $U(\rho,\theta)
\leq c$ by (\ref{hire}) we conclude that there exists $\tau \in
[0,d-\rho)$ such that
$$U(\rho+\tau,\theta)=c.$$
Hence we estimate
\begin{eqnarray*}
\bigg(\frac{\partial U(\rho,\theta)}{\partial
\rho}\bigg)^2&=&\bigg(\frac{\partial U(\rho+\tau,\theta)}{\partial
\rho}\bigg)^2 -\int_0^{\tau}\frac{d}{dt}\bigg(\frac{\partial
U(\rho+t,\theta)}{\partial \rho}\bigg)^2dt\\
&>&-2\int^{\tau}_0 \frac{\partial U(\rho+t,\theta)}{\partial
\rho}\frac{\partial^2 U(\rho+t,\theta)}{\partial \rho^2}dt\\
&\geq& 2\int^{\tau}_0 \frac{\partial U(\rho+t,\theta)}{\partial
\rho}dt \\
&=&2\Big(U(\rho+\tau,\theta)-U(\rho,\theta)\Big)\\
&=&2\Big(c-U(\rho,\theta)\Big).
\end{eqnarray*}
In the second step we have used Lemma~\ref{tra2} and in the third
step we have used again Lemma~\ref{tra2} together with
Lemma~\ref{tra3}. This implies (\ref{positiv}) and hence the
Proposition follows. \hfill $\square$

\begin{Remark}{\rm 
\label{remark:transversality_above_critical_value}
The above proof actually shows more than the claim of the proposition.
For $\epsilon>0$ small enough, the Liouville field is still transverse to level sets with energy $E<H(L_1)+\epsilon$ away from the Lagrange point $L_1$.

To see this, observe first of all that we can remove a ball around the Lagrange point $B_\delta(L_1)$ to divide the level set $H=E$ into a ``moon" and an ``earth" component.
We shall continue to consider the ``moon" component and argue that the ``moon" component is entirely contained in $B_d(M)$.
Indeed, recall that for fixed $\rho$ the function $U_\rho(\theta)$ attains its minimum at $U_\rho(0)$, so the distance to the moon is maximal along the ray $\theta=0$.
That means that, for $\epsilon$ small enough, the maximum distance to the moon $M$ occurs for small $\theta$, which we can assume to be less than $d$ provided that $\epsilon$ and $\delta$ are small enough.

With this observation in mind, we can repeat the argument of the proof of Proposition~\ref{transverse} to see that $X$ is still transverse to the level set $H=E$ for points $x$ that lie in the ``moon" component (so in particular outside $B_\delta(L_1)$).}
\end{Remark}

\section{Regularizing level sets of $H$ and Liouville vector fields}
The goal of this section is to show that we can perform Moser's regularization, and extend the Liouville vector field (\ref{eq:Liouville_field_unregularized}) to the regularization. We shall use the Hamiltonian
\begin{equation}
\label{eq:Hamiltonian_restricted_3body}
H(p,q)=\frac{1}{2} |p|^2 +p_1 q_2-p_2 q_1 -\frac{\mu}{|q-q^0|}-\frac{1-\mu}{|q-q^1|}.
\end{equation}
Here $q^0=( -(1-\mu),0)$ is the position of the moon, and $q^1=(\mu,0)$ the position of the earth in the rotating coordinate system.

Since the only problem in compactness of the components of the level sets corresponding to the earth or moon comes from points $q$ with either $|q-q^0|<\epsilon$ or $|q-q^1|<\epsilon$, we shall only consider that situation.
Without loss of generality, we shall assume that $|q-q^0|<\epsilon$.

\subsection{Transforming the Hamiltonian} Here we follow the notation introduced
in Section \ref{section:reg}.
In order to regularize $\Sigma$, we begin by reparametrizing the flow, i.e.~we put
$$
s=\int \frac{dt}{|q-q^0|}
$$
and introduce a new Hamiltonian by
$$
H=\frac{K}{|q-q^0|}+k.
$$
By examining Hamilton's equations we see that the flow of $K$ at energy level $0$ corresponds to the flow of $H$ at energy level $k$.
The new Hamiltonian is explicitely
$$
K=\frac{1}{2}(|p|^2+1)|q-q^0|+|q-q^0|(p_1q_2 -p_2 q_1) -(k+\frac{1}{2})|q-q^0|-\mu-(1-\mu)\left|\frac{q-q^0}{q-q^1}\right|.
$$
Now use the canonical transformation $p=-x$, $q-q^0=y$.
We get
$$
K=\frac{1}{2}(|x|^2+1)|y| + |y|(x_2 y_1-x_1 y_2)+ |y|(x_2 q^0_1-x_1 q^0_2)-(k+\frac{1}{2})|y| -(1-\mu)\frac{|y|}{|y+q^0-q^1|}-\mu
$$
\begin{Remark}{\rm 
This transformation reverses the roles of what one would usually expect from the position and momentum coordinates; this will become clear in the following.}
\end{Remark}

Perform the inverse of the stereographic projection given by (\ref{str1}), (\ref{str2}) and (\ref{str3}).
In particular, we have
$$
\epsilon>|q-q_0|=|y|=\frac{2|\eta|}{|x|^2+1}=\frac{2|\eta|}{\frac{1+\xi_0}{1-\xi_0}+\frac{1-\xi_0}{1-\xi_0}} =|\eta|(1-\xi_0).
$$
We get the new Hamiltonian
$$
\tilde K(\xi,\eta)=|\eta|\left(1+(1-\xi_0)(\xi_2\eta_1 -\xi_1\eta_2)+(\xi_2 q^0_1-\xi^1q^0_2)-(k+\frac{1}{2})(1-\xi_0) \right.
$$
$$
\left.
-(1-\mu)(1-\xi_0) \frac{1}{|\vec \eta (1-\xi_0)+\vec \xi \eta_0+q^0-q^1|}
\right)
-\mu.
$$
\noindent From this we get the following smooth Hamiltonian on $T^*S^2$ (shifting and squaring the previous Hamiltonian),
\begin{equation}
\label{eq:regularized_Ham_3body}
Q(\xi,\eta)=\frac{1}{2} |\eta|^2 f(\xi,\eta)^2,
\end{equation}
where
$$
f(\xi,\eta)=1+(1-\xi_0)(\xi_2\eta_1-\xi_1 \eta_2 )+(\xi_2q^0_1 -\xi_1 q^0_2) -(k+\frac{1}{2})(1-\xi_0)-\frac{(1-\mu)(1-\xi_0)}{|\vec \eta (1-\xi_0)+\vec \xi \eta_0+q^0-q^1|}
$$
Due to the shifting and squaring we now need to look at $\{Q=\frac{1}{2} \mu^2\}=\{K=0\}$.

\subsection{Level sets ``close" to the moon or the earth are contact}
We shall check that the Liouville vector field $X=\eta \partial_\eta $ is transverse to one connected component of the level set $\{Q=\frac{1}{2} \mu^2\}$ for points $(\xi,\eta)$ with $|\eta|(1-\xi_0) <\epsilon$.
Indeed,
$$
X(Q)=|\eta|^2 f(\xi,\eta)^2+|\eta|^2 f(\xi,\eta) \eta \partial_{\eta} f(\xi,\eta)=
$$
$$
2Q+|\eta|^2 f(\xi,\eta)
\left(
(1-\xi_0)(\xi_2\eta_1-\xi_1 \eta_2 )
-\eta \partial_\eta \frac{(1-\mu)(1-\xi_0)}{|\vec \eta (1-\xi_0)+\vec \xi \eta_0+q^0-q^1|}
\right)
$$
To estimate the last term, let us first make a few observations.
First of all, let us argue that we get an upper bound for $|\eta|$ from the equality $Q=\frac{1}{2} \mu^2$.
Indeed, the term $|f(\xi,\eta)|$ can be estimated as
$$
\left|
1+(1-\xi_0)(\xi_2\eta_1-\xi_1 \eta_2 )+(\xi_2q^0_1 -\xi_1 q^0_2) -(k+\frac{1}{2})(1-\xi_0)
-\frac{(1-\mu)(1-\xi_0)}{|\vec \eta (1-\xi_0)+\vec \xi \eta_0+q^0-q^1|}
\right| \geq
$$
$$
1-(1-\xi_0)|\xi| |\eta| -|\xi||q^0|+(1-\xi_0)(|k|-\frac{1}{2}-\frac{(1-\mu)}{|\vec \eta (1-\xi_0)+\vec \xi \eta_0+q^0-q^1|})
$$
Using $|q^0|=1-\mu$, $|\xi|=1$, and $(1-\xi_0)|\eta|<\epsilon$, we obtain
$$
|f(\xi,\eta)| \geq 1-\epsilon-(1-\mu)+(1-\xi_0)\left( |k|-\frac{1}{2}-\frac{(1-\mu)}{|\vec \eta (1-\xi_0)+\vec \xi \eta_0+q^0-q^1|} \right) \geq \frac{\mu}{2}
$$
In the last step we have argued that
$$
\frac{(1-\mu)}{|\vec \eta (1-\xi_0)+\vec \xi \eta_0+q^0-q^1|}=\frac{1-\mu}{|q-q^1|}\leq 2(1-\mu)
$$
is bounded on a small neighborhood of $q^0$ (we have $|y|=|q-q^0|<\epsilon$).
Indeed, choosing $\epsilon<1/2$ is enough for this.
In that case $|q-q^1|> 1/2$, so we see that the claim holds.
Choose the energy level $k$ appropriately.

As $Q=\frac{1}{2} \mu^2$, we obtain the promised upper bound for $|\eta|$:
$$
\frac{1}{2} \mu^2=Q \geq \frac{1}{2} |\eta|^2 \frac{\mu^2}{4},
$$
so we get
$$
|\eta|\leq 2.
$$
Let us then continue with our earlier computation to check that the level sets are contact. We may write
$$
X(Q)\geq 2Q-|\eta| \big(|\eta| |f(\xi,\eta)|\big) \eta \partial_\eta f(\xi,\eta).
$$
Observe that $|\eta|\leq 2$, $|\eta| |f(\xi,\eta)|=\sqrt{2Q}=\mu$. The remaining term can also be estimated.
$$
|\eta \partial_\eta f(\xi,\eta)|=\left|
(1-\xi_0)(\xi_2\eta_1-\xi_1 \eta_2 )
-\eta \partial_\eta \frac{(1-\mu)(1-\xi_0)}{|\vec \eta (1-\xi_0)+\vec \xi \eta_0+q^0-q^1|}
\right|
$$
$$
\leq
\epsilon+(1-\mu)\epsilon \partial_\eta \frac{1}{|\vec \eta (1-\xi_0)+\vec \xi \eta_0+q^0-q^1|}
$$
The latter term can obviously be bounded by some constant $C$ on a compact set away from the singular point.
As $|y|<\epsilon$, this holds true, so finally,
$$
X(Q)\geq \mu^2-2\mu \epsilon (1+(1-\mu)C).
$$
Now we see that for $\epsilon$ small enough we have $X(Q)>0$ on the connected component of the level set we are interested in.
\begin{Remark}
\label{remark:Liouville_is_natural}{\rm 
It is important to observe that the Liouville vector field $X$ is actually the natural Liouville vector field on the cotangent bundle, i.e.~$\eta\partial_\eta$.}
\end{Remark}

Combining this with Proposition~\ref{transverse} and the computations related to that proposition, we find the following.
\begin{Proposition}
\label{prop:level_sets_regularized}
Let $E<H(L_1)$ and consider the component $\Sigma$ of the level set $\{H=E\}$ corresponding to either $q^0$ or $q^1$.
Then $\Sigma$ can be regularized to form the closed $3$-manifold $\widetilde{\Sigma} \cong \R P^3$.
Furthermore, $\widetilde{\Sigma}$ carries a natural contact structure which is isomorphic to the standard contact structure on $\R P^3$.
Finally, the sublevel set $\{H\leq E\}$ provides a filling of $\Sigma$, which extends to a strong symplectic filling of $\widetilde{\Sigma}$.
The dynamics are preserved in this regularization, i.e.~Reeb orbits in $\Sigma$ correspond to Reeb orbits in $\widetilde{\Sigma}$.
\end{Proposition}

\begin{proof}
In order to combine Proposition~\ref{transverse} with the above computations, note that the Liouville vector field on the regularization and the one used in Proposition~\ref{transverse} are both the natural Liouville vector fields on the cotangent bundle, i.e.~both defined by the equation $i_X \omega_{can}=\lambda_{can}$.
This means in particular that they must coincide after the coordinate change.
Hence we obtain a global Liouville vector field for the regularized planar circular restricted three body problem.
This shows that the regularized level set $\widetilde{\Sigma}$ is contact, and diffeomorphic to $\R \P^3$, since the complete regularized level set is filled by the set $\{ Q\leq \frac{1}{2}\mu^2 \}$ which can be identified with $T^*_{\leq 1}S^2$.

Therefore we only need to check that we obtain the standard contact structure on $\widetilde{\Sigma}\cong \R P^3\cong ST^*S^2$.
This can be done directly, but in dimension $3$ we can alternatively argue as follows.
First of all, $\widetilde{\Sigma}$ is fillable, so it is tight by the Eliashberg-Gromov tightness theorem, see Theorem~\ref{thm:fillable_tight}.
On the other hand, we also know that $\R P^3$ has a unique tight contact structure by Theorem~\ref{thm:unique_tight_structure_RP3}.
Hence the contact structure on $\widetilde{\Sigma}$ induced by the Liouville field $X$ is isomorphic to the standard contact structure on $\R P^3$.
\end{proof}

\section{Connected sum}

\subsection{Interpolating Liouville vector fields}
Let us start by describing the setup. Write $q^0=\big(-(1-\mu,0)\big)$ and
$q^1=(\mu,0)$ for the location of the moon and earth, respectively.
Consider the Hamiltonian
$$
H(q,p):=\frac{1}{2}|p+Jq|^2+V_{\mathrm{eff}}(q),
$$
where
$$
J=\begin{pmatrix}0 &1\\-1&0 \end{pmatrix}
$$
and
$$
V_{\mathrm{eff}}=-\frac{1}{2}q^2-\frac{\mu}{|q-q^0|}-\frac{1-\mu}{|q-q^1|}.
$$
Following computations of Conley, \cite{conley}, we can expand the effective
potential around the Lagrange points $q^L$, which are the critical
points of $V_{\mathrm{eff}}$. We shall write
$$
V_{\mathrm{eff}}(q)=\widetilde{Q}(q-q^L)+R(q-q^L),
$$
where $\widetilde{Q}$ is a quadratic form in $q-q^L$ given by
$$
\widetilde{Q}=\frac12 \left(
\begin{array}{cc}
-2 \rho&0\\
0&\rho
\end{array}
\right) ,
$$
and $R(q-q^L)$ is a rest term of higher order than $(q-q^L)^2$.

We now expand the Hamiltonian $H(q,p)$ around the critical point
$(q,p)=(q^L,-Jq^L)$. We obtain
$$
H(q,p)=\frac{1}{2}|p+Jq|^2+\tilde
Q(q-q^L)+R(q-q^L)=\frac{1}{2}|(p+Jq^L)+J(q-q^L)|^2+\tilde
Q(q-q^L)+R(q-q^L),
$$
which we shall rewrite as
$$
H(q,p)=Q(q-q^L,p+Jq^L)+R(q-q^L),
$$
using the quadratic form $Q$ given by
$$
Q=\frac12\left(
\begin{array}{cccc}
-2 \rho&0&0&1\\
0&\rho&-1&0\\
0&-1&1&0\\
1&0&0&1
\end{array}
\right) .
$$
Hence we consider the Hamiltonian
$$
Q(q,p)+R(q),
$$
for $q,p$ near $0$. We have the Weinstein-like vector field
$$
Y_{a,b}= (q_1,q_2,p_1,p_2) \left(
\begin{array}{cccc}
a & 0 & 0 & 0 \\
0 & b & 0 & 0 \\
0 & 0 & 1-a & 0 \\
0 & 0 & 0 & 1-b
\end{array}
\right) \left(
\begin{array}{c}
\partial_{q_1} \\
\partial_{q_2} \\
\partial_{p_1} \\
\partial_{p_2}
\end{array}
\right) .
$$
Then we have the following lemma.
\begin{Lemma}
The vector field $Y_{a,b}$ is Liouville. Furthermore, for a sufficiently
small neighborhood of $(q^L,-Jq^L)$ there exists $\epsilon>0$, $a<0$ and $b>0$ such
that for $E\in[H(L_1)-\epsilon,H(L_1)+\epsilon]-\{ H(L_1) \}$ the vector field
$Y:=Y_{a,b}$ is transverse to the level set $H=E$.
\end{Lemma}
\begin{proof}
The Liouville property can be directly checked.
Furthermore, note that $Y(Q)$ is a quadratic form in $(q,p)$.
For appropriate values of $a$ and $b$ (for instance $a=-1$ and $b=1/2$), we can directly check that $Y(Q)$ is positive definite.
In order to see this, we may use the following expression, which is taken from Conley \cite{conley}  right below equation (4) therein,
\begin{equation}
\label{eq:rho_Hessian_linearized_kepler_problem}
\rho=\frac{\mu}{|q-q^0|^3}+\frac{1-\mu}{|q-q^1|^3}.
\end{equation}
Using equation \eqref{eq:mu_Lagrange_point}, we can easily check that $\rho\geq 4$.
After that, one can verify that the eigenvalues of the quadratic form $Y(Q)$ are positive.

On the other hand, $Y(R)$ is
still of order $q^2$, so for a sufficiently small neighborhood of
$(q^L,-Jq^L)$ we can estimate this quantity by a multiple of $Y(Q)$,
$$
Y(Q+R)=Y(Q)+Y(R) \geq Y(Q)-1/2|Y(Q)|>0.
$$
\end{proof}
The contact form induced by $Y$ is given by
$$
\alpha_1=-a(q_1-q_1^L)dp_1-bq_2 dp_2+(1-a)p_1 dq_1+
(1-b)(p_2-q_1^L)dq_2.
$$
On the other hand, by Remark~\ref{remark:transversality_above_critical_value} the vector field $X=(q-q^1)\partial_q$ is
Liouville and transverse to $\{ (q,p) | H(q,p)=E \}$, for $(q,p)$
away from the critical point $L_1$. 

From $X$ we obtain the induced 1-form
$$
\alpha_0=(q_1^1-q_1)dp_1-q_2dp_2.
$$
Note that 
$$
\alpha_1-\alpha_0=(1-a)\left( (q_1-q_1^L)dp_1+p_1dq_1\right)
+(q_1^L-q_1^1)dp_1+(1-b)\left( q_2 dp_2+(p_2-q_1^L)dq_2 \right) .
$$
is exact with primitive
$$
G:=(1-a)(q_1-q_1^L)p_1+(q_1^L-q_1^1)p_1+(1-b)(p_2-q_1^L)q_2.
$$
From now on we use the more convenient to use coordinates around the critical
point $(q^L,-Jq^L)$, i.e.~$(q',p')=(q-q^L,p+Jq^L)$. With some abuse of notation
(leaving out $'$), we obtain
$$
G=(1-a)q_1p_1+(q_1^L-q_1^1)p_1+(1-b)p_2q_2.
$$
If we denote by $X_i$, $i=0,1$, the Liouville vector field corresponding to the 1-form $\alpha_i$ and define the Hamiltonian vector field $X_G$ of $G$ by
$$
i_{X_G}\omega=dG,
$$
then we have
$$
X_1=X_0+X_{G}.
$$
Next we choose a cut-off function $f$ depending on $q_1-\frac{1}{\rho}p_2$ to
interpolate between $X_0$ for $q_1$ large and $X_1$ for $q_1$ close
to $0$ and make the following ansatz
$$
X=X_0+X_{f G}.
$$

\begin{Remark}
\label{remark:cutoff_divides_level_set}\rm
The dependence of $f$ on the specific linear combination $q_1-\frac{1}{\rho} p_2$ is chosen to cancel out some terms in a later computation.
As we explain next the set $\{q_1-\frac{1}{\rho} p_2=0\}$ divides the singular energy hypersurface $\{H=H(L_1)\}$ into two connected components.
We shall call the component containing $q^0$ in its closure the moon component and the component containing $q^1$ the earth component.

Indeed, this can be seen as follows.
First of all, the singular energy hypersurface $\{H=H(L_1)\}$ corresponds to $\{Q=0\}$.
We now claim that the hyperplane $\{q_1-\frac{1}{\rho} p_2=\delta\}$ intersects the quadric $\{Q=0\}$ in a $2$-sphere that collapses to a point for $\delta=0$.
Indeed, inserting the equation $q_1-\frac{1}{\rho}p_2=\delta$ into $Q=0$ yields
$$
\frac{1}{2}\left(
(p_1-q_2)^2+((\rho+1)q_1-\rho \delta)^2+(\rho-1)q_2^2-(2\rho+1)q_1^2
\right) =0,
$$
which can simplified to an equation describing a sphere,
$$
(p_1-q_2)^2+(\rho q_1-(\rho+1)\delta)^2+
(\rho-1)q_2^2
=(2\rho+1)\delta^2.
$$
\end{Remark}

Let us now check that the Liouville field $X$ is transverse to level sets of $H$ close to the Lagrange point.
By choosing a cut-off function $f$ taking values between $0$ and $1$ and vanishing outside a small neighborhood of the Lagrange point $q^L$, we see that $X$ coincides with $X_0$ far away from $q^L$.
We checked earlier that then $X_0(H)>0$ away from the Lagrange point, see Remark~\ref{remark:transversality_above_critical_value}.
Recall also the definition of the Poisson bracket,
$$
\{ F,G \}:=\omega(X_F,X_G)=dF(X_G).
$$
Then we see
\begin{eqnarray*}
X(H)&=&X_0(H)+X_{fG}(H)=dH(X_0)+\{ H,fG \}=dH(X_0)-\{fG,H\}\\
&=& (1-f)dH(X_0)+fdH(X_0)-f\{G,H\}-G\{f,H \} \\
&=& (1-f)dH(X_0)+fdH(X_0+X_G)+GdH(X_f).
\end{eqnarray*}
Note that $X_0+X_G=Y$, so the first two  terms are non-negative near the Lagrange point.
For the last term, note that the Hamiltonian vector field of $f$ is given by
$$
X_f=f'\cdot\Big(\frac{\partial}{\partial p_1}+\frac{1}{\rho}\frac{\partial}{\partial q_2}\Big).
$$
We directly compute $dH(X_f)$. We obtain
$$
dH(X_f)=f'\cdot\left
( \Big(1-\frac{1}{\rho}\Big)p_1+\frac{q_2}{\rho}
\Big(\frac{\mu}{|q-q^1|^3}+\frac{1-\mu}{|q-q^1|^3}-\rho\Big)
\right)
.
$$
Observe that the latter term vanishes in $q^L$ by the equation \eqref{eq:rho_Hessian_linearized_kepler_problem} for $\rho$.
As a result, we see that $dH(X_f)\sim f'(1-\frac{1}{\rho})p_1$, up to higher order terms.

Now we specify the cut-off function.
Choose small constants $0<\epsilon_1<\epsilon_2$.
Choose a smooth function $f$ such that $f=1$ on $[0,\epsilon_1]$ and $f=0$ on $[\epsilon_2,\infty)$, and with non-positive derivative in between.
For later purposes we extend $f$ symmetrically round $0$.

Then observe that the leading order term in $G$ is $(q_1^L-q_1^1)p_1$ and the coefficient $q_1^L-q_1^1$  is negative.
Hence the product $GdH(X_f)$ has a leading order term equal to
$$
f'\cdot(q_1^L-q_1^1)\Big(1-\frac{1}{\rho}\Big)p_1^2,
$$
which is non-negative. In particular we see that the leading order terms in $X(H)$ are all non-negative and at least one term is positive. Thus, they dominate the higher order terms.
Hence we see that $X$ is a Liouville vector field that is everywhere transverse to a component of the level set $H=H(L_1)+\epsilon$ provided that $\epsilon$ is small enough.

We have now all ingredients to prove the following lemma.
\begin{Lemma}
There exists $\epsilon>0$ and a Liouville vector field $\widetilde{X}$ such that $\widetilde{X}$ is transverse to level sets $H=E$ for all $E_L<E<E_L+\epsilon$. 
\end{Lemma}

\begin{proof}
As observed in Remark~\ref{remark:cutoff_divides_level_set} the linear functional $q_1-\frac{1}{\rho}p_2$, on which $f$ depends, divides the level set $\{H=E\}$ into a moon and earth component.

In the above construction we already constructed a Liouville vector that is transverse to the earth component.
Since the constructed vector field is equal to $X_1$ on a small neighborhood of the Lagrange point, where $f=1$, the same construction on the moon side will coincide on this ball with the Liouville vector field for the earth component, so we end up with a globally defined Liouville vector field, that is transverse to the entire level set $\{H=E\}$.
\end{proof}

\subsection{Contact structures on different level sets of $H$}
In the previous section we have shown that all level sets of $H$ for energy less than $H(L_1)+\epsilon$ are of restricted contact type. 
On the other hand, Proposition~\ref{prop:level_sets_regularized} shows that level sets of $H$ with energy less than $H(L_1)$ can be regularized to form a closed manifold, namely $\R P^3$ with its natural contact structure, while preserving the dynamics.
We pointed out in Remark~\ref{remark:Liouville_is_natural} that the Liouville vector field used for the regularization is the natural Liouville vector field for the cotangent bundle, so it automatically coincides with the Liouville vector field $(q-q^0)\partial_{q}$ after a suitable coordinate change.

Now fix an energy level $H(L_1)<E<H(L_1)+\epsilon$.
The above observation that the Liouville field is natural applies to both the earth and moon component of the level set $H=E$.
Hence we can regularize the level set $\Sigma=\{ H=E \}$ both near $q^0$ and near $q^1$ in order to obtain a closed manifold $\widetilde{\Sigma}$.

We see directly that topologically $\widetilde{\Sigma}\cong \widetilde{\Sigma}_1 \# \widetilde{\Sigma}_2$ by observing that $\{ H(L_1)-\epsilon \leq H \leq E \}$ provides a cobordism $W$ between $\Sigma_1 \cup \Sigma_2$ and $\Sigma$.
This cobordism admits $H$ as Morse function with only one index $1$ critical point, namely $L_1$. Thus, we can regard $\Sigma$ and hence $\widetilde{\Sigma}$ as a connected sum.

\begin{Remark}\label{rmk:bdry_connected_sum}\rm
We observe that $\widetilde{\Sigma}$ is symplectically fillable. The above cobordism argument shows that, topologically, the filling is given by the boundary connected sum $T^*_{\leq 1}S^2 \natural T^*_{\leq 1}S^2$. The set $\{ H(L_1)-\epsilon \leq H \leq E \}\cap B_\delta(L_1)$ is a symplectic handle corresponding to a $1$-handle attachment as described in Section~\ref{sec:symplectic_handles}. A small isotopy between our setup, which is very close to the Weinstein model, and the actual Weinstein model induces the isotopy required to show that as contact manifolds
$$
\widetilde{\Sigma}\cong \widetilde{\Sigma}_1 \# \widetilde{\Sigma}_2.
$$
Alternatively, we can reason that, since $\widetilde{\Sigma}$ is fillable, it is tight.
Now apply Theorem~\ref{thm:unique_prime_decomposition}, which asserts that tight contact manifolds have a unique prime decomposition.
This immediately implies that $\widetilde{\Sigma} \cong (\R P^3,\xi_1)\# (\R P^3,\xi_2)$, where $\xi_1$ and $\xi_2$ are tight contact structures on $\R P^3$.
Now we can simply use the fact that there is a unique tight contact structure on $\R P^3$ up to isotopy just as in Proposition~\ref{prop:level_sets_regularized} to complete the proof of Theorem~A.
\end{Remark}


\appendix
\section{Computations}

Here are some detailed computations omitted in the main text of Section \ref{section:below}.

$$U(\rho,\theta)=-\frac{\mu}{\rho}-\frac{1-\mu}{\sqrt{\rho^2-2\rho \cos
\theta+1}}-\frac{1}{2}\rho^2+\rho\cos\theta(1-\mu)-\frac{1}{2}(1-\mu)^2.$$
$$\frac{\partial U}{\partial \theta}=(1-\mu)\rho\sin\theta
\bigg(\frac{1}{\big(\rho^2-2\rho\cos\theta+1\big)^{\frac{3}{2}}}-1\bigg).$$
\begin{eqnarray*}
\frac{\partial U}{\partial
\rho}=\frac{\mu}{\rho^2}+(1-\mu)\frac{\rho-\cos
\theta}{(\rho^2-2\rho\cos\theta+1)^{\frac{3}{2}}}-\rho+\cos
\theta(1-\mu).
\end{eqnarray*}
\begin{eqnarray*}
\frac{\partial^2 U}{\partial \rho \partial \theta} &=& (1-\mu)
\Bigg(\frac{\sin\theta\big(\rho^2-2\rho\cos\theta+1\big)^{\frac{3}{2}}-
\frac{3}{2}(\rho^2-2\rho\cos\theta+1)
^{\frac{1}{2}}2\rho\sin\theta(\rho-\cos\theta)}{\big(\rho^2-2\rho\cos\theta
+1\big)^3}-\sin\theta\Bigg)\\
&=&(1-\mu)\sin\theta\Bigg(\frac{\rho^2-2\rho\cos\theta+1-3\rho(\rho-\cos\theta)}
{\big(\rho^2-2\rho\cos\theta+1\big)^{\frac{5}{2}}}-1\Bigg)\\
&=&(1-\mu)\sin\theta\Bigg(\frac{-2\rho^2+\rho\cos\theta+1}
{\big(\rho^2-2\rho\cos\theta+1\big)^{\frac{5}{2}}}-1\Bigg)
\end{eqnarray*}
\begin{eqnarray*}
\frac{\partial^2 U}{\partial \rho \partial \rho}&=&
-\frac{2\mu}{\rho^3}+(1-\mu)\frac{(\rho^2-2\rho \cos
\theta+1)^{\frac{3}{2}}-\frac{3}{2}(\rho^2-2\rho\cos
\theta+1)^{\frac{1}{2}}(2\rho-2\cos\theta)(\rho-\cos\theta)}
{(\rho^2-2\rho\cos\theta+1)^3}-1\\
&=&-\frac{2\mu}{\rho^3}+(1-\mu)\frac{\rho^2-2\rho\cos\theta+1-3(\rho-\cos\theta)^2}
{(\rho^2-2\rho\cos\theta+1)^{\frac{5}{2}}}-1\\
&=&-\frac{2\mu}{\rho^3}+\frac{1-\mu}
{(\rho^2-2\rho\cos\theta+1)^{\frac{5}{2}}}\Big(\rho^2-2\rho\cos\theta+1-3\rho^2
+6\rho\cos\theta-3\cos^2\theta\Big)-1\\
&=&-\frac{2\mu}{\rho^3}-\frac{1-\mu}
{(\rho^2-2\rho\cos\theta+1)^{\frac{5}{2}}}\Big(2\rho^2-4\rho\cos\theta-1+3\cos^2\theta\Big)-1
\end{eqnarray*}
\begin{eqnarray*}
\frac{\partial^3 U}{\partial \rho^3}&=&\frac{6\mu}{\rho^4}-(1-\mu)
\frac{(4\rho-4\cos \theta)(\rho^2-2\rho\cos
\theta+1)^{\frac{5}{2}}}{(\rho^2-2\rho\cos\theta+1)^5}\\
& &+(1-\mu)\frac{
\frac{5}{2}(\rho^2-2\rho\cos\theta+1)^{\frac{3}{2}}(2\rho-2\cos\theta)
(2\rho^2-4\rho\cos\theta-1+3\cos^2\theta)}{(\rho^2-2\rho\cos\theta+1)^5}\\
&=&\frac{6\mu}{\rho^4}-(1-\mu)(\rho-\cos\theta)
\frac{4(\rho^2-2\rho\cos\theta+1)-5(2\rho^2-4\rho\cos\theta-1+3\cos^2\theta)}
{(\rho^2-2\rho\cos\theta+1)^{\frac{7}{2}}}\\
&=&\frac{6\mu}{\rho^4}-(1-\mu)(\rho-\cos\theta)\frac{-6\rho^2+12\rho\cos\theta
+9-15\cos^2\theta}{(\rho^2-2\rho\cos\theta+1)^{\frac{7}{2}}}\\
&=&\frac{6\mu}{\rho^4}-3(1-\mu)(\rho-\cos\theta)\frac{3+4\rho\cos\theta
-2\rho^2-5\cos^2\theta}{(\rho^2-2\rho\cos\theta+1)^{\frac{7}{2}}}
\end{eqnarray*}
\begin{eqnarray*}
U(\rho,0)-U(\rho,\pi)&=&
(1-\mu)\bigg(-\frac{1}{1-\rho}+\rho+\frac{1}{1+\rho}+\rho\bigg)\\
&=&(1-\mu)\frac{-1-\rho+1-\rho+2\rho(1-\rho)(1+\rho)}{(1-\rho)(1+\rho)}\\
&=&(1-\mu)\frac{-2\rho+2\rho-2\rho^3}{(1-\rho)(1+\rho)}\\
&=&-\frac{2(1-\mu)\rho^3}{(1-\rho)(1+\rho)}\\
&<&0.
\end{eqnarray*}
\begin{eqnarray*}
\frac{\partial U}{\partial \rho}(\rho,0)-\frac{\partial U}{\partial
\rho}(\rho,\pi)&=&(1-\mu)\Bigg(\frac{\rho-1}{(1-\rho)^3}+1-
\frac{\rho+1}{(1+\rho)^3}+1\Bigg)\\
&=&(1-\mu)\Bigg(2-\frac{1}{(1-\rho)^2}-\frac{1}{(1+\rho)^2}\Bigg)\\
&=&(1-\mu)\frac{2(1-\rho^2)^2-(1+\rho)^2-(1-\rho)^2}{(1-\rho)^2(1+\rho)^2}
\\
&=&(1-\mu)\frac{2-4\rho^2+2\rho^4-1-2\rho-\rho^2-1+2\rho-\rho^2}
{(1-\rho)^2(1+\rho)^2}\\
&=&2(1-\mu)\frac{\rho^4-3\rho^2}{(1-\rho)^2(1+\rho)^2}
\end{eqnarray*}
$$g(\theta)=\frac{2\rho^2-4\rho\cos\theta-1+3\cos^2\theta}
{(\rho^2-2\rho\cos\theta+1)^{\frac{5}{2}}}.$$
$$\kappa(\theta)=\rho^2-2\rho\cos\theta+1.$$
\begin{eqnarray*}
\frac{dg}{d\theta}&=&\frac{(4\rho\sin\theta-6\cos\theta\sin\theta)\kappa^{\frac{5}{2}}
-\frac{5}{2}2\rho\sin\theta\kappa^{\frac{3}{2}}(2\rho^2-4\rho\cos\theta-1+3
\cos^2\theta)}{\kappa^5}\\
&=&\frac{\sin\theta}{\kappa^{\frac{7}{2}}}\bigg((4\rho-6\cos\theta)
(\rho^2-2\rho\cos\theta+1)
-5\rho(2\rho^2-4\rho\cos\theta-1+3\cos^2\theta)\bigg)\\
&=&\frac{\sin\theta}{\kappa^{\frac{7}{2}}}\bigg(4\rho^3-8\rho^2\cos\theta
+4\rho-6\rho^2\cos\theta+12\rho\cos^2\theta-6\cos\theta\\
& &\,\,\,\,\,\,\,\,\,\,\,\,\,\,\,-10\rho^3+
20\rho^2\cos\theta+5\rho-15\rho\cos^2\theta\bigg)\\
&=&\frac{\sin\theta}{\kappa^{\frac{7}{2}}}\bigg(-6\rho^3+6\rho^2\cos\theta
-3\rho\cos^2\theta+9\rho-6\cos\theta\bigg)\\
&=&-\frac{3\sin\theta}{\kappa^{\frac{7}{2}}}\bigg(\rho\cos^2\theta
+2(1-\rho^2)\cos\theta+\rho(2\rho^2-3)\bigg)
\end{eqnarray*}
$$\theta \in M.$$
\begin{eqnarray*}
\cos\theta&=&\frac{2(\rho^2-1)\pm\sqrt{4(1-\rho^2)^2-4\rho^2(2\rho^2-3)}}
{2\rho}\\
&=&\frac{\rho^2-1\pm\sqrt{1-2\rho^2+\rho^4-2\rho^4+3\rho^2}}{\rho}\\
&=&\frac{\rho^2-1\pm\sqrt{-\rho^4+\rho^2+1}}{\rho}\\
&=&\frac{\rho^2-1+\sqrt{-\rho^4+\rho^2+1}}{\rho}
\end{eqnarray*}
\begin{eqnarray*}
g(\theta)&=&\frac{2\rho^2-4\rho^2+4-4\sqrt{-\rho^4+\rho^2+1}-1
+\frac{3}{\rho^2}\Big((\rho^2-1)^2+2(\rho^2-1)\sqrt{-\rho^4+\rho^2+1}-\rho^4
+\rho^2+1\Big)}{\big(\rho^2-2\rho^2+2-2\sqrt{-\rho^4+\rho^2+1}+1\big)^{\frac{5}{2}}}\\
&=&\frac{-2\rho^4+3\rho^2-4\rho^2\sqrt{-\rho^4+\rho^2+1}
+3\rho^4-6\rho^2+3+6(\rho^2-1)\sqrt{-\rho^4+\rho^2+1}-3\rho^4
+3\rho^2+3}{\big(3-\rho^2-2\sqrt{-\rho^4+\rho^2+1}\big)^{\frac{5}{2}}
\rho^2}\\
&=&\frac{6-2\rho^4-(6-2\rho^2)\sqrt{-\rho^4+\rho^2+1}}{\big(3-\rho^2-2\sqrt{-\rho^4+\rho^2+1}\big)^{\frac{5}{2}}
\rho^2}
\end{eqnarray*}

\begin{eqnarray*}
3+4\rho\cos\theta-2\rho^2-5\cos^2\theta
&=&3+4\rho^2-4+4\sqrt{-\rho^4+\rho^2+1}-2\rho^2\\
& &-\frac{5}{\rho^2}
\Bigg(\rho^4-2\rho^2+1+2(\rho^2-1)\sqrt{-\rho^4+\rho^2+1}-\rho^4+\rho^2+1\Bigg)\\
&=&2\rho^2-1+4\sqrt{-\rho^4+\rho^2+1}\\
& &-\frac{5}{\rho^2}
\Bigg(-\rho^2+2+2(\rho^2-1)\sqrt{-\rho^4+\rho^2+1}\Bigg)\\
&=&2\rho^2-1+4\sqrt{-\rho^4+\rho^2+1}+5-\frac{10}{\rho^2}\\
&
&-10\sqrt{-\rho^4+\rho^2+1}+\frac{10}{\rho^2}\sqrt{-\rho^4+\rho^2+1}\\
&=&2\rho^2+4-6\sqrt{-\rho^4+\rho^2+1}-\frac{10}{\rho^2}+\frac{10}{\rho^2}
\sqrt{-\rho^4+\rho^2+1}\\
&=& \frac{2\rho^4+4\rho^2-10 +(10-6\rho^2)
\sqrt{-\rho^4+\rho^2+1}}{\rho^2}\\
&=&\frac{(10-6\rho^2)^2(-\rho^4+\rho^2+1)-(2\rho^4+4\rho^2-10)^2}
{\rho^2\Big(10-2\rho^4-4\rho^2+(10-6\rho^2)\sqrt{-\rho^4+\rho^2+1}\Big)}\\
&=&\frac{(36\rho^4-120\rho^2+100)(-\rho^4+\rho^2+1)}
{\rho^2\Big(10-2\rho^4-4\rho^2+(10-6\rho^2)\sqrt{-\rho^4+\rho^2+1}\Big)}\\
& &-\frac{
(4\rho^8+8\rho^6-20\rho^4+8\rho^6+16\rho^4-40\rho^2-20\rho^4-40\rho^2+100)}
{\rho^2\Big(10-2\rho^4-4\rho^2+(10-6\rho^2)\sqrt{-\rho^4+\rho^2+1}\Big)}\\
&=&\frac{-36\rho^8+120\rho^6-100\rho^4+36\rho^6-120\rho^4+100\rho^2
+36\rho^4-120\rho^2+100}
{\rho^2\Big(10-2\rho^4-4\rho^2+(10-6\rho^2)\sqrt{-\rho^4+\rho^2+1}\Big)}\\
& &+\frac{ -4\rho^8-16\rho^6+24\rho^4+80\rho^2-100}
{\rho^2\Big(10-2\rho^4-4\rho^2+(10-6\rho^2)\sqrt{-\rho^4+\rho^2+1}\Big)}\\
&=&\frac{-40\rho^8+140\rho^6-160\rho^4+60\rho^2}
{\rho^2\Big(10-2\rho^4-4\rho^2+(10-6\rho^2)\sqrt{-\rho^4+\rho^2+1}\Big)}\\
&=&\frac{20\big(3-8\rho^2+7\rho^4-2\rho^6\big)}
{10-2\rho^4-4\rho^2+(10-6\rho^2)\sqrt{-\rho^4+\rho^2+1}}
\end{eqnarray*}

$$U(r,s)=-\frac{\mu}{\sqrt{r^2+s^2}}-\frac{1-\mu}{\sqrt{(1-r)^2+s^2}}
-\frac{1}{2}(r-1+\mu)^2-\frac{1}{2}s^2.$$
$$\frac{\partial U}{\partial
r}=\frac{r\mu}{\big(r^2+s^2\big)^{\frac{3}{2}}}-
\frac{(1-r)(1-\mu)}{\big((1-r)^2+s^2\big)^{\frac{3}{2}}}-(r-1+\mu).$$
$$\frac{\partial U}{\partial
s}=\frac{s\mu}{\big(r^2+s^2\big)^{\frac{3}{2}}}+
\frac{s(1-\mu)}{\big((1-r)^2+s^2\big)^{\frac{3}{2}}}-s.$$
\begin{eqnarray*}
r\frac{\partial U}{\partial r}+s\frac{\partial U}{\partial s}
&=&\frac{(r^2+s^2)\mu}{\big(r^2+s^2\big)^{\frac{3}{2}}}+
\frac{(1-\mu)}{\big((1-r)^2+s^2\big)^{\frac{3}{2}}}\big(r(r-1)+s^2\big)
+r(1-r-\mu)-s^2\\
&=&\frac{\mu}{\big(r^2+s^2\big)^{\frac{1}{2}}}+
\frac{(1-\mu)}{\big((1-r)^2+s^2\big)^{\frac{3}{2}}}\big(r^2+s^2-r\big)
+r(1-r-\mu)-s^2
\end{eqnarray*}
$$(p_1+s)^2+(p_2+1-r-\mu)^2=2(k-U).$$
\begin{eqnarray*}
& &\frac{\mu}{\sqrt{r^2+s^2}}-p_2r+p_1s+
\frac{1-\mu}{\big((1-r)^2+s^2\big)^{\frac{3}{2}}}(r^2+s^2-r)\\
&=&\frac{\mu}{\sqrt{r^2+s^2}}-(p_2+1-r-\mu)r+r(1-r-\mu)
+(p_1+s)s-s^2\\
& &+\frac{1-\mu}{\big((1-r)^2+s^2\big)^{\frac{3}{2}}}(r^2+s^2-r)\\
&\geq&\frac{\mu}{\sqrt{r^2+s^2}}+r(1-r-\mu)
-s^2+\frac{1-\mu}{\big((1-r)^2+s^2\big)^{\frac{3}{2}}}(r^2+s^2-r)\\
& &-\sqrt{r^2+s^2}\sqrt{(p_1+s)^2+(p_2+1-r-\mu)^2}\\
&=&r\frac{\partial U}{\partial r}+s\frac{\partial U}{\partial s}
-\sqrt{r^2+s^2}\sqrt{2(k-U)}.
\end{eqnarray*}

$$\rho^5-(3-\mu)\rho^4+(3-2\mu)\rho^3-\mu\rho^2+2\mu\rho-\mu=0.$$
$$\mu=-\frac{\rho^5-3\rho^4+3\rho^3}{\rho^4-2\rho^3-\rho^2+2\rho-1}.$$
\begin{eqnarray*}
& &-\frac{2\mu}{\rho^3}-(1-\mu)g(\theta)\\
&=&\frac{2\rho^3\big(\rho^2-3\rho+3\big)}
{\big(\rho^4-2\rho^3-\rho^2+2\rho-1\big)\rho^3}\\
& &-
\frac{\big(\rho^4-2\rho^3-\rho^2+2\rho-1+\rho^5-3\rho^4+3\rho^3\big)
\big(6-2\rho^4-(6-2\rho^2)\sqrt{-\rho^4+\rho^2+1\big)}}
{\big(\rho^4-2\rho^3-\rho^2+2\rho-1\big)
\big(3-\rho^2-2\sqrt{-\rho^4+\rho^2+1}\big)^{\frac{5}{2}} \rho^2}
\\
&=&\frac{2\rho^2\big(\rho^2-3\rho+3\big)
\big(3-\rho^2-2\sqrt{-\rho^4+\rho^2+1}\big)^{\frac{5}{2}}}
{\big(\rho^4-2\rho^3-\rho^2+2\rho-1\big)
\big(3-\rho^2-2\sqrt{-\rho^4+\rho^2+1}\big)^{\frac{5}{2}} \rho^2}\\
& &-\frac{\big(\rho^5-2\rho^4+\rho^3-\rho^2+2\rho-1\big)
\big(6-2\rho^4-(6-2\rho^2)\sqrt{-\rho^4+\rho^2+1}\big)}
{\big(\rho^4-2\rho^3-\rho^2+2\rho-1\big)
\big(3-\rho^2-2\sqrt{-\rho^4+\rho^2+1}\big)^{\frac{5}{2}} \rho^2}.
\end{eqnarray*}

\subsection{Proof that $W(\rho_0,\theta_0)$ is negative}
Observe that
$$
d^5-2d^4+d^3-d^2+2d-1=(d^2+d+1)(d-1)^3,
$$
and that
$$
d^4-2d^3-d^2+2d-1=(d-1/2)^4-5/2(d-1/2)^2-7/16.
$$
The latter can be checked to be negative on the interval $[0,1]$ We
now rewrite everything to see that this expression is negative. The
main trick which we shall repeat over and over again, is the simple
identity
$$
a-\sqrt{b}=\frac{a^2-b}{a+\sqrt b}.
$$
For instance, we see that
$$
3-d^2-2\sqrt{1+d^2-d^4}=\frac{(3-d^2)^2-4(1+d^2-d^4)}{3-d^2+2\sqrt{1+d^2-d^4}
}=\frac{5(1-d)^2(1+d)^2}{3-d^2+2\sqrt{1+d^2-d^4}}.
$$
Hence we obtain the following:
\begin{eqnarray*}
&=&\frac{2d^2\big(d^2-3d+3\big)
\big(3-d^2-2\sqrt{-d^4+d^2+1}\big)^{\frac{5}{2}}}
{\big(d^4-2d^3-d^2+2d-1\big)
\big(3-d^2-2\sqrt{-d^4+d^2+1}\big)^{\frac{5}{2}} d^2}\\
& &-\frac{\big(d^5-2d^4+d^3-d^2+2d-1\big)
\big(6-2d^4-(6-2d^2)\sqrt{-d^4+d^2+1}\big)}
{\big(d^4-2d^3-d^2+2d-1\big)
\big(3-d^2-2\sqrt{-d^4+d^2+1}\big)^{\frac{5}{2}} d^2}
\\
&=&\frac{ 2(d^2-3d+3) \left(
\frac{5(1-d)^2(1+d)^2}{3-d^2+2\sqrt{1+d^2-d^4} } \right)^{5/2}
}{\left( (d-1/2)^4-5/2(d-1/2)^2-7/16 \right) \left(
\frac{5(1-d)^2(1+d)^2}{3-d^2+2\sqrt{1+d^2-d^4} } \right)^{5/2}
}\\
& &-\frac{(d^2+d+1)(d-1)^3 \left(
-6\frac{(1-d)(1+d)}{1+\sqrt{1+d^2-d^4}}+2\sqrt{1+d^2-d^4}-2d^2
\right)} {\left( (d-1/2)^4-5/2(d-1/2)^2-7/16 \right) \left(
\frac{5(1-d)^2(1+d)^2}{3-d^2+2\sqrt{1+d^2-d^4} } \right)^{5/2} }
\end{eqnarray*}
Note here that $d\in[0,1]$, so we can pull out $(1-d)^5$ in the
expression above. We obtain
\begin{eqnarray*}
&=&\frac{ 2(d^2-3d+3) (1-d)^5(1+d)^5
\frac{5^{5/2}}{(3-d^2+2\sqrt{1+d^2-d^4})^{5/2} } }{\left(
(d-1/2)^4-5/2(d-1/2)^2-7/16 \right) (1-d)^5(1+d)^5
\frac{5^{5/2}}{(3-d^2+2\sqrt{1+d^2-d^4})^{5/2} }
}\\
&&+\frac{
2(d^2+d+1)(1-d)^3\frac{(1-d)(1+d)}{1+\sqrt{1+d^2-d^4}}(-2+\sqrt{1+d^2-d^4}+d^2)
} {\left( (d-1/2)^4-5/2(d-1/2)^2-7/16 \right) (1-d)^5(1+d)^5
\frac{5^{5/2}}{(3-d^2+2\sqrt{1+d^2-d^4})^{5/2} } }
\\
&=& 2\left( \frac{ (d^2-3d+3) (1-d)^5(1+d)^5
\frac{5^{5/2}}{(3-d^2+2\sqrt{1+d^2-d^4})^{5/2} } }{\left(
(d-1/2)^4-5/2(d-1/2)^2-7/16 \right) (1-d)^5(1+d)^5
\frac{5^{5/2}}{(3-d^2+2\sqrt{1+d^2-d^4})^{5/2} } } \right.
\\
&& \left. +\frac{ \frac{(d^2+d+1)(1-d)^5(1+d)^2 \left(
-1+\frac{d^2}{1+\sqrt{1+d^2-d^4}} \right) } {1+\sqrt{1+d^2-d^4}} }
{\left( (d-1/2)^4-5/2(d-1/2)^2-7/16 \right) (1-d)^5(1+d)^5
\frac{5^{5/2}}{(3-d^2+2\sqrt{1+d^2-d^4})^{5/2} } } \right)
\\
&=& 2 \frac{5^{5/2}(1+d)^3 \frac{d^2-3d+3}
{(3-d^2+2\sqrt{1+d^2-d^4})^{5/2}} +
\frac{(d^2+d+1)d^2}{(1+\sqrt{1+d^2-d^4})^2} - \frac{(d^2+d+1)}
{1+\sqrt{1+d^2-d^4}} } {\left( (d-1/2)^4-5/2(d-1/2)^2-7/16 \right)
(1+d)^3 \frac{5^{5/2}}{(3-d^2+2\sqrt{1+d^2-d^4})^{5/2}} }
\end{eqnarray*}
Now observe that the denominator is negative for $d\in[0,1]$; there
are no zeroes for $d\in[0,1]$ and the function
$(d-1/2)^4-5/2(d-1/2)^2-7/16$ is clearly negative. Furthermore, the
denominator is positive for the following reason. We shall estimate
the sum of the first and last term,
\begin{eqnarray*}
& & \frac{5^{5/2}(1+d)^3(d^2-3d+3)}
{(3-d^2+2\sqrt{1+d^2-d^4})^{5/2}} - \frac{(d^2+d+1)}
{1+\sqrt{1+d^2-d^4}}
\geq  \\
& & (d^2+d+1)\left( \frac{5^{5/2}} {(3-d^2+2\sqrt{1+d^2-d^4})^{5/2}}
- \frac{1} {1+\sqrt{1+d^2-d^4}} \right) .
\end{eqnarray*}
We see that this estimate holds, since for $d\in [0,1]$ we have
$d^2+d+1 \leq (1+d)^3$, and $d^2-3d+3\geq 1$; indeed,
$(d^2-3d+3)'=2d-3$ is negative on $[0,1]$, so $d^2-3d+3$ is minimal
in $d=1$. Finally observe that $\frac{1} {1+\sqrt{1+d^2-d^4}}<1$,
whereas
$$
\frac{5^{5/2}} {(3-d^2+2\sqrt{1+d^2-d^4})^{5/2}}\geq 1.
$$
The latter can be seen by observing that $(3-d^2+2\sqrt{1+d^2-d^4})$
has derivative
$$
-2d \left(1-\frac{1-2d^2}{\sqrt{1+d^2-d^4}} \right) \leq
-2d(\frac{2d^2}{\sqrt{1+d^2-d^4}})\leq 0,
$$
so
$$
\frac{5^{5/2}} {(3-d^2+2\sqrt{1+d^2-d^4})^{5/2}}
$$
is minimal for $d=0$, where its value is 1. This shows the claim.


\begin{thebibliography}{99}

\bibitem{Abbo_Schwarz_On_the_Floer_homology_of_cotangent_bundles}
A.~Abbondandolo and M.~Schwarz, \emph{On the Floer homology of cotangent bundles}, Comm. Pure Appl. Math. {\bf 59} (2006) 254--316.


 \bibitem{abraham-marsden} R.\,Abraham, J.\,Marsden,
 \emph{Foundations of Mechanics}, 2nd ed. (Addison-Wesley, New York,
 1978).
 
 \bibitem{Albers_Frauenfelder_Leafwise_intersections_and_RFH}
P.~Albers and U.~Frauenfelder, \emph{Leaf-wise intersections and Rabinowitz Floer homology}, J. Topol. Anal.  {\bf 2} (2010) 77--98.


 
 \bibitem{Cieliebak_Handle_attaching_in_symplectic_homology_and_the_chord_conjecture}
K.~Cieliebak, \emph{Handle attaching in symplectic homology and the chord conjecture}, J. Eur. Math. Soc. (JEMS) {\bf 4} (2002) 115--142.

\bibitem{Cieliebak_Frauenfelder_Restrictions_to_displaceable_exact_contact_embeddings}
K.~Cieliebak and U.~Frauenfelder, \emph{A Floer homology for exact contact embeddings}, Pacific J. Math. {\bf 293} (2009) 251--316.



\bibitem{Cieliebak_Frauenfelder_Oancea_Rabinowitz_Floer_homology_and_symplectic_homology}
K.~Cieliebak, U.~Frauenfelder, and A.~Oancea, \emph{Rabinowitz Floer homology and symplectic homology}, 2009, arXiv:0903.0768, to appear in Annales Scientifiques de L'ENS.



\bibitem{CFP} K. Cieliebak, U. Frauenfelder, G.P. Paternain, \emph{Symplectic topology of Ma\~n\'e's critical values,} Geom. Topol. {\bf 14} (2010) 1765--1870. 

 
 \bibitem{conley} C.~C.~Conley, \emph{Low energy transit orbits in the restricted three-body problem.}, SIAM J. Appl. Math. 16 1968 732--746. 

\bibitem{CMP} G. Contreras, L. Macarini, G.P. Paternain, \emph{Periodic orbits for exact magnetic flows on surfaces,} Int. Math. Res. Not. {\bf 8} (2004) 361--387.

\bibitem{Ding_Geiges} F. Ding, H. Geiges, \emph{A unique decomposition theorem for tight contact 3-manifolds,} Enseign. Math. {\bf 53} (2007) 333--345.

 \bibitem{eliashberg} Y.~Eliashberg, \emph{Lectures on symplectic topology in Cala Gonone. Basic notions, problems and some methods}, Conference on Differential Geometry and Topology (Sardinia, 1988).  Rend. Sem. Fac. Sci. Univ. Cagliari  58  (1988),  suppl., 27--49.
 \bibitem{gromov} M.~Gromov, \emph{Pseudoholomorphic curves in symplectic manifolds},  Invent. Math.  82  (1985),  no. 2, 307--347.

\bibitem{hofer}  H. Hofer, \emph{Pseudoholomorphic curves in symplectizations with applications to the Weinstein conjecture in dimension three,} Invent. Math. {\bf 114} (1993) 515--563.

\bibitem{McLean_Lefschetz_fibrations_and_symplectic_homology}
M.~McLean, \emph{Lefschetz fibrations and symplectic homology},
Geom. Topol. {\bf 13} (2009) 1877--1944.

\bibitem{moser}J. Moser, \emph{Regularization of Kepler's problem and the averaging method on a manifold}, Comm. Pure Appl. Math. {\bf 23} (1970) 609--636. 

\bibitem{Moser_A_fixed_point_theorem_in_symplectic_geometry}
J.~Moser, \emph{A fixed point theorem in symplectic geometry}, Acta Math. {\bf 141} (1978) 17--34.

\bibitem{rab} P.H. Rabinowitz, \emph{Periodic solutions of a Hamiltonian system on a prescribed energy surface,} J. Differential Equations {\bf 33} (1979) 336--352.

\bibitem{Salamon_Weber_Floer_homology_and_heat_flow}
D.~A. Salamon and J.~Weber, \emph{Floer homology and the heat flow},
Geom. Funct. Anal. {\bf 16} (2006)1050--1138.



\bibitem{Viterbo_Symplectic_topology_as_the_geometry_of_generating_functions}
C.~Viterbo, \emph{Symplectic topology as the geometry of generating functions},
Math. Ann. {\bf 292} (1992) 685--710.





\end{thebibliography}
\end{document}